\documentclass[]{interact}

\usepackage{algorithm}%
\usepackage{algorithmicx}%
\usepackage{algpseudocode}%

\usepackage{epstopdf}
\usepackage[caption=false]{subfig}

\usepackage{enumerate}%
\usepackage{enumitem}
\usepackage{url}
\usepackage{tikz}
\usetikzlibrary{spy}

\usepackage[numbers,sort&compress]{natbib}
\bibpunct[, ]{[}{]}{,}{n}{,}{,}
\renewcommand\bibfont{\fontsize{10}{12}\selectfont}

\theoremstyle{plain}
\newtheorem{theorem}{Theorem}[section]
\newtheorem{lemma}[theorem]{Lemma}
\newtheorem{corollary}[theorem]{Corollary}
\newtheorem{proposition}[theorem]{Proposition}

\theoremstyle{definition}
\newtheorem{definition}[theorem]{Definition}

\theoremstyle{remark}
\newtheorem{remark}{Remark}

\usepackage[capitalize]{cleveref}

\crefname{table}{Table}{Table}
\crefname{lemma}{Lemma}{Lemma}
\crefname{proposition}{Proposition}{Proposition}

\newcommand{\vectorizeGR}[1]{\bm{#1}}

\newcommand{\vlambda}{\vectorizeGR{\lambda}}
\newcommand{\pnp}{PnPSplit$^{+}$}

\DeclareMathOperator{\amin}{argmin}

\DeclareMathOperator{\prox}{prox}
\DeclareMathOperator{\proj}{proj}

\usepackage{pgffor}
\foreach \a in {q,w,e,r,t,y,u,i,o,p,a,s,d,f,g,h,j,k,l,z,x,c,v,b,n,m,Q,W,E,R,T,Y,U,I,O,P,A,S,D,F,G,H,J,K,L,Z,X,C,V,B,N,M}{ 
	\expandafter\xdef\csname v\a\endcsname {{\noexpand\mathbf{\a}}}
}

\def\lp{\left(}
\def\rp{\right)}
\def\lq{\left[}
\def\rq{\right]}

\def\la{\left\langle}
\def\ra{\right\rangle}
\newcommand{\Id}{\vI_{\rm {\bm d}}}

\newcommand{\argmin}[1]{\ensuremath{\underset{\substack{{#1}}}{\amin}}\,\,}

\newcommand{\kiter}[1]{#1^{k}}
\newcommand{\kkiter}[1]{#1^{k+1}}
\newcommand{\vwk}{\kiter{\vw}}
\newcommand{\vwkk}{\kkiter{\vw}}
\newcommand{\vxk}{\kiter{\vx}}
\newcommand{\vxkk}{\kkiter{\vx}}

\newcommand{\vlambdak}{\kiter{\vlambda}}
\newcommand{\vlambdakk}{\kkiter{\vlambda}}
\newcommand{\den}{D_{\bm{\vartheta}}}

\newcommand{\rd}{\mathbb{R}}
\newcommand{\ds}{\displaystyle}

\newcommand{\gsdsplit}{GSDSplit$^+\,$}
\newcommand{\pnpsplit}{PnPSplit$^+\,$}

\begin{document}


\title{A Relaxed Gradient Step Denoiser for Splitting Methods in Poisson Inverse Problems}

\author{
\name{A. Benfenati\textsuperscript{a,b}\thanks{CONTACT A. Benfenati. Email: alessandro.benfenati@unimi.it}}
\affil{\textsuperscript{a}Environmental and Science Policy Department, Università degli studi di Milano, Via Celoria 2, 20133, Milano, Italia}
	 \affil{\textsuperscript{b}GNCS, INDAM, Viale Aldo Moro 5, 00185 Roma, Italia}
}

\maketitle

\begin{abstract}
Plug\&Play methods combine classical variational models with learned denoisers and have achieved strong results in imaging inverse problems. Their convergence has been widely studied for Gaussian data, whereas Poisson models require additional care because of the nonquadratic Kullback--Leibler fidelity. This work introduces \gsdsplit, a splitting method built upon \pnpsplit, in which the generic denoising block is replaced by a relaxed Gradient Step Denoiser. The resulting method retains explicit Poisson fidelity updates. The denoiser is defined through the gradient of a learned convex potential parameterized by an Input Convex Neural Network. Its architecture and training promote both blind denoising capability and a smoothness condition sufficient for firm nonexpansiveness. Empirical smoothness estimates are used to select a relaxation parameter compatible with the convergence assumptions of \pnpsplit. Numerical experiments provide evidence that the trained denoiser operates in an FNE-compatible regime on the considered validation data. \gsdsplit also yields stable reconstructions outside the empirically supported range and shows substantially lower sensitivity to the ADMM parameter, while remaining effective for different blur operators and Poisson noise levels.
\end{abstract}


\section{Introduction}

Image restoration is central in scientific imaging, including microscopy \cite{8901171,zunino2023reconstructing}, astronomy \cite{bertero2021introduction,aghabiglou2024r2d2}, and medical imaging \cite{morotti2026adaptive,9732379,Loreto02012025}. The image acquisition process is  described by the linear model \cite{bertero2021introduction,10.1088/2053-2563/aae109}
\begin{equation}
	\label{eq:linModel}
	\vg = \mathcal{N}(\vH\,\vx^\star + b),
\end{equation}
where $\vx^\star\in\rd^n$ is the clean data, $\vg\in\rd^m$ is the recorded data, $\vH\in\rd^{m \times n}$ is the linear blur operator, $b\in\rd^+$ is a constant background term, and $\mathcal{N}$ models the statistical noise. The standard assumptions on the operator $\vH$ are $\vH_{i,j} \geq 0\,  \forall i,j, \quad \vH^\top\bm{1} = \bm{1}$, where $\bm{1}$ is the vector whose components are all equal to 1. These assumptions are commonly satisfied in practical applications, such as in microscopy and astronomy. The noise model $\mathcal{N}$ depends on the physics underlying the acquisition process: the most common one is Gaussian noise, but several applications also involve speckle noise \cite{cheng2024diffusion}, salt\&pepper  or Cauchy \cite{BENFENATI2026986} noise.

Image restoration problems aim to recover an estimate of $\vx^\star$ from $\vg$ given $\vH$, $b$, and the noise model. This work focuses on images corrupted by Poisson noise, which is  signal-dependent and naturally arises in photon-limited acquisition systems, including fluorescence microscopy, astronomy, and emission tomography. Adopting a maximum a posteriori approach (MAP), an estimation of $\vx^\star$ is obtained by solving \cite{10.1088/2053-2563/aae109}
\begin{equation}
	\label{eq:varpb}
	\argmin{\vx\in\rd^n_+} KL(\vH\,\vx+b,\vg) + \beta\,R(\vx),
\end{equation} 
where $KL$ is the generalized Kullback--Leibler functional
\begin{equation}
	\label{eq:KL}
	KL(\vH\,\vx+b,\vg) = \sum_i \lq g_i\log\lp\frac{g_i}{(\vH\vx)_i+b}\rp + (\vH\vx)_i+b-g_i\rq,
\end{equation}
and it has the role of the data fidelity function. In \eqref{eq:KL} the convention $0\log(0)=0$ is adopted. The constraint enforces nonnegativity on the reconstructed image, which is natural since pixel values represent light intensities. The function $R$ is a regularizer that controls the effect of noise and promotes desired image features, such as sharp edges or sparsity. The parameter $\beta\in\rd^+$ is called the regularization parameter and weights the influence of $R$ on the objective functional.

The optimization problem \eqref{eq:varpb} can be solved in several ways. The Alternating Di-
rection Method of Multipliers \cite{boyd2011distributed,Gao03032020} remains popular and has led to several variants,
including PIDAL \cite{figueiredo2010restoration}, PIDSPLIT+ \cite{Setzer2012} and weighted TV approaches \cite{9732379}. First or- der methods have been widely explored and successfully employed even in High Performance Computing settings \cite{zanella2013towards}. Among first-order approaches, one can find also proximal gradient methods \cite{chouzenoux2024variational}.

In \cite{6737048}, the authors observed that in an ADMM scheme the proximal step related to the regularization function $R$ can be considered as a Gaussian denoising step: therefore, they proposed a Plug And Play scheme (PnP), which consists of replacing the proximal step with a Gaussian denoiser $D$. Classical choices are the Block Matching and 3D filtering (BM3D) \cite{dabov2007image}, the non-local means filter (NLM) \cite{buades2011non} or a denoising convolutional neural network (CNN) \cite{7839189}. The convergence analysis of several PnP schemes requires specific assumptions on the denoiser, such as nonexpansiveness or firm nonexpansiveness (FNE). In the FNE case, the denoiser can be interpreted as the resolvent of a maximally monotone operator.

In the era of deep learning and neural networks, one possible strategy to satisfy this is to train a CNN that is FNE \cite{doi:10.1137/20M1387961}. Other approaches include controlling the Lipschitz constant of the
denoiser \cite{ducotterd2024improving}, and considering constrained variants of
ADMM \cite{jimaging10020050}. The investigation of PnP methods led to Regularization by Denoising (RED) theory \cite{romano2017little}, which shifted the requirement on the Jacobian of $D$. The RED theoretical framework is solid and well-established; nevertheless it presents some practical issues,  that have been addressed and partially solved under a constrained framework (RED-PRO) \cite{cohen2021regularization}: RED provides an explicit denoiser-induced regularization under assumptions on the denoiser Jacobian, while RED-PRO instead adopts a constrained approach considering the fixed-point set of the denoiser. Another constrained approach (CRED) for the RED framework has been addressed in \cite{10415495}, using a discrepancy measure between the recorded data $\vg$ and the estimated solution. Recently, Deep Equilibrium Models have been employed to tackle the theoretical issues of PnP methods from a different point of view \cite{daniele2026deep}.

PnP methods have mainly focused on image restoration problems where $\vg$ is perturbed by Gaussian noise, with the fidelity function being the least-square functional. Poisson data instead require using the KL function, as shown in \eqref{eq:varpb}: this functional is not quadratic and its gradient may become problematic when the estimated intensities
are close to zero. In \cite{ROND201696} a classical ADMM approach is employed: the presence of the KL functional leads to the use an iterative procedure for solving the optimization step related to such functional. The authors of \cite{9616253} introduce a quantum-based denoising scheme for Poisson deconvolution, and \cite{hurault2023convergent} shows how the Euclidean geometry is not the most suitable one in the presence of Poisson noise, while the one induced by Bregman divergences is more descriptive. Bayesian Poisson reconstruction has also been addressed through Langevin-based sampling methods in \cite{klatzer2026efficient}. Primal-Dual schemes, which avoid the inversion of linear operators, are deeply studied in \cite{11299501,7936537}.

The work \cite{benfenati2025plug} introduced the \pnpsplit method, which couples the PnP approach with  PIDSPLIT+. It employs an ADMM scheme, according to \cite{Setzer2012}; then,  the proximity operator of the regularization function is replaced by a Gaussian CNN denoiser trained to be FNE, following \cite{doi:10.1137/20M1387961}, yielding a convergence result \cite[Proposition 2]{benfenati2025plug}. 

This work proposes \gsdsplit, a structured variant of \pnpsplit in which the generic PnP denoising block is replaced by a Gradient Step Denoiser \cite{hurault2022gradient} of the form
\begin{equation}
	\label{eq:GSD}
	D_\vartheta(\vx) = \vx - \nabla \psi_\vartheta(\vx),	
\end{equation}
The potential $\psi_\vartheta$ is parametrized by an input convex neural network (ICNN) \cite{pmlr-v70-amos17b,zhang2026rethinking} and trained through a residual Gaussian denoising loss, together with an empirical Hessian spectral penalty promoting $L$-smoothness with $L$ close to 1. Under the ideal assumptions that $\psi_\vartheta$ is convex and $L$-smooth with $L\leq1$, the operator $D_\vartheta$ is FNE, and therefore fits the convergence framework of \pnpsplit.  Gradient Step Denoisers form a structured class of denoisers of the form \eqref{eq:GSD}: their residual $\Id-D_\vartheta$ is a conservative vector field, because it is the gradient of a learned potential. This property connects GSDs with variational regularization, monotone operator theory and fixed-point formulations \cite{hurault2022gradient}. When $\psi_\vartheta$ is convex and smooth, the resulting operator inherits useful nonexpansiveness properties, making GSDs particularly suitable for PnP algorithms requiring theoretical assumptions on denoising operators, such as \pnpsplit. Additionally, this work considers the more general formulation that includes a
relaxation parameter $\alpha$:
\begin{equation}
	\label{eq:GSD2}
	D_{\vartheta,\alpha}(\vx) = \vx - \alpha\nabla \psi_\vartheta(\vx).
\end{equation}
Such a choice allows to employ L-smooth network with $L>1$, since the FNE condition for $D_{\vartheta,\alpha}$ is satisfied when $\alpha\,L\leq 1$. This makes it possible to employ more general ICNNs in the proposed splitting scheme. 

Moreover, beyond this compatibility with splitting approaches, the GSD structure makes the denoising residual an explicit gradient field, providing a natural bridge between splitting-based PnP schemes and fixed-point formulations such as RED-PRO. This connection will be investigated in future work.

The main contributions of this work are the following.

\gsdsplit is introduced by embedding an ICNN-based Gradient Step Denoiser into the \pnpsplit scheme for Poisson image restoration. Second, a convergence result is obtained by relating the FNE property of the denoiser to the convexity and smoothness of the learned $\psi_\vartheta$ and to the relaxation parameter $\alpha$. Third, a smooth ICNN architecture and a tailored training strategy are implemented to develop a denoiser operator satisfying the desired properties. Finally, the method is numerically assessed in terms of denoising ability, sampled FNE conditions, convergence residuals, reconstruction quality, computational cost, and sensitivity to the
ADMM parameter. The experiments show that \gsdsplit attains lower peak reconstruction quality than \pnpsplit when both methods are optimally tuned, but is substantially less sensitive to an overestimate of the ADMM parameter.

\medskip

This work is organized as follows. \cref{sec:GSDSplit} presents the novel approach, building on previous literature on PnP methods. \cref{sec:ICNN} presents the developed denoiser, with details on its architecture and the training phase. \cref{sec:num} assesses the performance of the proposed method, its robustness with respect to its hyperparameters, and a comparison with splitting methods. Finally, \cref{sec:conc} summarizes the presented approach and outlines directions for future research.

\paragraph*{Notations.} $\rd^n$ denotes the real vector space, $\rd^n_+$ is the nonnegative orthant of $\rd^n$. $\rd^{m\times n}$ denotes the real vector space of real matrices with $m$ rows and $n$ columns. Bold capital letters denote matrices, whereas bold lowercase letters denote vectors. The notation $\|\vx\|_p$ denotes the $p$-norm of the vector $\vx$. Italic and Greek letters denote real constants. The operator $\proj_A$ is the projection on the set $A$. The indicator function $\iota_A(x)$ is 0 if $x\in A$, $+\infty$ otherwise. The proximal operator,  also called proximity operator or simply prox, of a function $f$ at a point $\vz$ is
$$
\prox_f(\vz) = \argmin{\vx} f(\vx) +\frac12\|\vx-\vz\|_2^2.
$$

\section{\gsdsplit: a Gradient-Step-Denoising variant of \pnpsplit}
\label{sec:GSDSplit}

This section introduces the proposed \gsdsplit method. Firstly, the \pnpsplit scheme for Poisson image restoration and its convergence requirement on the denoising operator is presented. Then,  the generic denoiser is replaced by a Gradient Step Denoiser, and the investigation focuses on the condition under which the resulting method remains within the \pnpsplit convergence framework.

\subsection{PnP Splitting for Poisson Inverse Imaging}
Poisson image restoration problems are addressed by solving \eqref{eq:varpb}, which is recast  below encompassing the constraint in the objective via the indicator function:
$$
	\argmin{\vx\in\rd^n} KL(\vH\,\vx+b,\vg) + \beta\,R(\vx) + \iota_{\rd^n_+}(\vx),
$$
Following \cite{Setzer2012}, define $\vM=\lp\vH^\top, \Id, \Id\rp^\top$ such that $
\vM\vx = \vw$, with $\vw=\lp \vw_1^\top-b, \vw_2^\top, \vw_3^\top\rp^\top$. Therefore the problem is recast in
$$
\argmin{\vw,\vx}  KL(\vw_1,\vg) + \beta\,R(\vw_2) + \iota_{\rd^n_+}(\vw_3), \qquad \text{s.t. } \vM\vx=\vw.
$$
Introducing the augmented Lagrangian function of this problem and after some algebraic manipulation, this yields the saddle-point problem
$$
\min_{\vw,\vx}\max_{\vlambda}  KL(\vw_1,\vg) + \beta\,R(\vw_2) + \iota_{\rd^n_+}(\vw_3) + \frac{1}{2\gamma}\|\vM\vx-\vw+\vlambda\|_2^2 - \frac{1}{2\gamma}\|\vlambda\|_2^2,
$$
where $\vlambda$ is the scaled Lagrangian multiplier. The PIDSPLIT+ method in \cref{al:pidsplit+} iteratively approaches a solution to the problem: the convergence results can be found in \cite{Setzer2012}. 
\begin{algorithm}[htbp]
	\caption{PIDSPLIT+\cite{Setzer2012}}
	\label{al:pidsplit+}
	\begin{algorithmic}
		\State{Set $\vx^0, \vw^0$ and accordingly $\vlambda^0$; select the parameter $\gamma>0$.}
		\For{$k=0,1,\dots$}
		\State{$\vxkk = \lp\vH^\top\vH + 2\Id\rp^{-1}\lq\vH^\top\lp\vwk_1-b-\vlambdak_1\rp + \vwk_2-\vlambdak_2 + \vwk_3-\vlambda_3^k\rq$}
		\State{}
		\State{$ \vwkk_1 =\ds\frac12\lp\ \vH\vxkk+\vlambdak_1-\gamma +\sqrt{\lp\vH\vxkk+\vlambdak_1-\gamma\rp^2+4\gamma\,\vg}\rp$}
		\State{}
		\State{$\vwkk_2 = \prox_{\gamma \beta R} \lp\vxkk+\vlambdak_2\rp$}
		\State{}
		\State{$\vwkk_3 = \proj_{\rd^n_+} \lp\vxkk+\vlambdak_3\rp$}
		\State{}
		\State{$ \vlambdakk =\vlambdak +\vM\vxkk-\vwkk $}
		\EndFor
	\end{algorithmic}
\end{algorithm}
The main advantage of this splitting is that the KL-related subproblem admits a componentwise closed-form solution, avoiding inner iterative solvers. The update step in $\vx$ has one and only one solution, because $\vH^\top\vH+2\Id\succ0$. When $\vH$ is a convolution operator with suitable boundary conditions, this linear system can be solved efficiently by FFT.

The update of $\vw_2$ relies on the proximity operator of $R$, therefore one needs \textit{i)} to select a regularization function, such as $\ell_1$ norm or Total Variation, and \textit{ii)} compute  the proximity operator. 
Adopting a PnP approach, in \cite{benfenati2025plug} the update on $\vw_2$ is performed via a Gaussian denoiser, denoted by $\mathcal{D}_{\gamma\beta}$, properly trained to be FNE: this led to the \pnpsplit method, reported in \cref{al:pnpsplit+}. The definition of firmly nonexpansive operators is reported for the sake of completeness.
\begin{definition}
	\label{def:FNE}
	Let $\mathcal{H}$ be a Hilbert space and let
	$T:\mathcal{H}\to\mathcal{H}$. The operator $T$ is firmly
	nonexpansive if
	$$
	\|T(x)-T(y)\|_2^2 \leq \la x-y, T(x)-T(y)\ra
	$$
	for every $x,y\in\mathcal{H}$.
\end{definition}
\cref{al:pnpsplit+} exploits the main advantages of PIDSPLIT+ and PnP approaches: from the former, it inherits the explicit formula for both $\vxkk$ and $\vwkk_1$, hence the deblurring step can be computed efficiently. On the other hand, adopting a PnP framework avoids the need to choose an explicit regularizer $R$ and to derive or compute its proximal operator. When the denoiser is FNE, then it is the resolvent of a maximally monotone operator: this property allows to state a convergence result for \pnpsplit, which is reported below.
\begin{algorithm}[ht]
	\caption{\pnp}
	\label{al:pnpsplit+}
	\begin{algorithmic}
		\State{Set $\vx^0, \vw^0$ and accordingly $\vlambda^0$; select the parameter $\gamma>0$.}
		\For{$k=0,1,\dots$}
		\State{$\vxkk = \lp\vH^\top\vH + 2\Id\rp^{-1}\lq\vH^\top\lp\vwk_1-b-\vlambdak_1\rp + \vwk_2-\vlambdak_2 + \vwk_3-\vlambda_3^k\rq$}
		\State{}
		\State{$ \vwkk_1 =\ds\frac12\lp\ \vH\vxkk+\vlambdak_1-\gamma +\sqrt{\lp\vH\vxkk+\vlambdak_1-\gamma\rp^2+4\gamma\,\vg}\rp,$}
		\State{}
		\State{$\vwkk_2 = \mathcal{D}_{\beta\gamma}\lp\vxkk+\vlambdak_2\rp$}
		\State{}
		\State{$\vwkk_3 = \proj_{\rd^n_+} \lp\vxkk+\vlambdak_3\rp$}
		\State{}
		\State{$ \vlambdakk =\vlambdak +\vM\vxkk-\vwkk $}
		\EndFor
	\end{algorithmic}
\end{algorithm}
\begin{proposition} 
		\label{prp:convpnp}\cite[Prop. 2]{benfenati2025plug} Let $\den$ be a firmly nonexpansive Gaussian denoiser, and assume that it is the resolvent of a maximally monotone operator $A$. Set $\mathcal{D}_{\gamma\beta}=\den$ in \cref{al:pnpsplit+}. For any $\vx^0, \vw^0$ and for any $\gamma\in\rd^+$ the sequences $\{\vlambdak\}_k$ and $\{\vwk\}_k$ generated by \pnp converge. The sequence $\{\vxk\}_k$ computed in \pnp converges to $\tilde\vx$ such that
		$$
		0\in\vH^\top\nabla KL(\vH\tilde\vx+b, \vg) + A(\tilde\vx) + N_{\rd_+^n}(\tilde \vx),
		$$ 
		where $N_{\rd_+^n}$ is the normal cone to $\rd^n_+$, if either of the following conditions holds:
		\begin{enumerate}[label=\roman*)]
			\item The primal problem has one and only one solution
			\item The optimization problem
			\begin{equation}\label{eq:mineProp}
				\argmin{\vx} \frac{1}{2\gamma}\|\vM\vx-\hat\vw +\hat\vlambda \|_2^2, \quad 
				\hat\vw = \lim_{k\to\infty} \vwk, \quad \hat\vlambda = \lim_{k\to\infty} \vlambdak 
			\end{equation}
		\end{enumerate}
		has a unique solution.
\end{proposition}

\subsection{GSD embedded into \pnpsplit}
Gradient step denoisers, introduced in \cite{hurault2022gradient} are structured operators involving the gradient of a scalar potential. In this work the following class of GSDs is considered:
\begin{equation}
	\label{eq:GSDrel}
	D_{\vartheta,\alpha}(\vx) = \vx -\alpha \nabla\psi_\theta(\vx),
\end{equation}
where $\psi_\vartheta$ is a learned potential. For $\alpha \in(0,1)$, 
$$
D_{\vartheta,\alpha}
=
(1-\alpha)\Id+\alpha D_{\vartheta,1},
$$
and therefore the denoising correction is relaxed. 

In classical PnP approaches, a denoiser is generally a black box removing Gaussian noise, while in the GSD approach, instead,  its residual has an explicit structure $\vx-D_{\vartheta,\alpha}(\vx) = \alpha\nabla\psi_\theta(\vx)$, therefore $\alpha\nabla\psi_\theta(\vx)$ is indeed the learned denoising correction. The vector $\nabla\psi_\vartheta(\vx)$ is trained to predict the denoising residual, and \eqref{eq:GSDrel} can equivalently be interpreted as one gradient-descent step on the learned potential
$\psi_\vartheta$.

Recent work \cite{zhang2026rethinking} has established a connection between ICNN-based GSDs and pseudo-contractive denoisers. In the present work, this more general framework is restricted to the firmly nonexpansive regime required by the convergence theory of \pnpsplit. More precisely, the potential $\psi_\vartheta$ is parameterized by an input convex neural network, while its smoothness is controlled during training through an empirical Hessian spectral penalty.

\begin{lemma}
	\label{lem:FNE}
	Let $\psi_\vartheta:\rd^n \to \rd$ be  convex and $L$-smooth. Then
	$$
	D_{\vartheta,\alpha}(\vx) = \vx-\alpha\nabla\psi_\vartheta(\vx)
	$$
	is firmly nonexpansive for any $0<\alpha\leq \ds\frac{1}{L}$.
\end{lemma}

\begin{proof}
	\newcommand{\nx}{\nabla\psi_\vartheta(\vx)}
	\newcommand{\ny}{\nabla\psi_\vartheta(\vy)}
	If $\psi_\vartheta$ is convex and $L$-smooth, then its gradient is $1/L$-cocoercive:
	$$
	\la\nabla\psi_\theta(\vx)- \nabla\psi_\theta(\vy),\vx-\vy\ra\geq \frac1L\|\nabla\psi_\theta(\vx)-\nabla\psi_\theta(\vy)\|_2^2
	$$
	and if $\alpha\leq\frac1L$
	$$
	\alpha\la\nabla\psi_\theta(\vx)- \nabla\psi_\theta(\vy),\vx-\vy\ra\geq \alpha^2\|\nabla\psi_\theta(\vx)-\nabla\psi_\theta(\vy)\|_2^2
	$$
	Then
	\begin{eqnarray*}
		\|D_{\vartheta,\alpha}(\vx)-D_{\vartheta,\alpha}(\vy)\|_2^2 &=& \|\vx-\alpha\nx - (\vy-\alpha\ny)\|_2^2 \\
		&=&\|\vx- \vy\|_2^2 + \|\alpha\nx - \alpha\ny\|_2^2\\ &&+2\la\vx-\vy,\alpha\ny-\alpha\nx\ra\\
		&=&\|\vx- \vy\|_2^2 + \|\alpha\nx - \alpha\ny\|_2^2\\ &&-2\la\vx-\vy,\alpha\nx-\alpha\ny\ra\\
		&\leq&\|\vx- \vy\|_2^2 + \alpha\la\nx- \ny,\vx-\vy\ra\\ &&-2\la\vx-\vy,\alpha\nx-\alpha\ny\ra\\
		&=&\|\vx- \vy\|_2^2  -\la\vx-\vy,\alpha\nx-\alpha\ny\ra\\
		&=&\|\vx- \vy\|_2^2  -\la\vx-\vy,\vx-\vx+\alpha\nx+\vy-\vy-\alpha\ny\ra\\
		&=&\la\vx-\vy,D_{\vartheta,\alpha}(\vx)-D_{\vartheta,\alpha}(\vy)\ra.
	\end{eqnarray*}
	and hence $D_{\vartheta,\alpha}$ is FNE.
\end{proof} 
A denoiser of the form \eqref{eq:GSDrel} satisfying Lemma~\ref{lem:FNE} can be embedded in the splitting scheme of \cref{al:pnpsplit+}: this new approach is shown in \cref{al:gsdsplit+}.
\begin{algorithm}[ht]
	\caption{\gsdsplit}
	\label{al:gsdsplit+}
	\begin{algorithmic}
		\State{Set $\vx^0, \vw^0$ and accordingly $\vlambda^0$; select the parameter $\gamma>0$. Let $\psi_\vartheta$ be a trained potential and $0<\alpha\leq\,L^{-1}$ satisfying hypotheses in Lemma~\ref{lem:FNE}.}
		\For{$k=0,1,\dots$}
		\State{$\vxkk = \lp\vH^\top\vH + 2\Id\rp^{-1}\lq\vH^\top\lp\vwk_1-b-\vlambdak_1\rp + \vwk_2-\vlambdak_2 + \vwk_3-\vlambda_3^k\rq$}
		\State{}
		\State{$ \vwkk_1 =\ds\frac12\lp\ \vH\vxkk+\vlambdak_1-\gamma +\sqrt{\lp\vH\vxkk+\vlambdak_1-\gamma\rp^2+4\gamma\,\vg}\rp,$}
		\State{}
		\State{$\vwkk_2 = \vxkk+\vlambdak_2 - \alpha\nabla\psi_\vartheta(\vxkk+\vlambdak_2)$}
		\State{}
		\State{$\vwkk_3 = \proj_{\rd^n_+} \lp\vxkk+\vlambdak_3\rp$}
		\State{}
		\State{$ \vlambdakk =\vlambdak +\vM\vxkk-\vwkk $}
		\EndFor
	\end{algorithmic}
\end{algorithm}
The convergence result follows from the assumptions on the denoiser.
\begin{corollary}
	Let $\psi_\vartheta$ be a function satisfying the hypothesis in Lemma~\ref{lem:FNE}. Then \gsdsplit converges to  $\tilde\vx$ such that
	$$
	0\in\vH^\top\nabla KL(\vH\tilde\vx+b, \vg) + G_{\vartheta,\alpha}(\tilde\vx) + N_{\rd_+^n}(\tilde \vx),
	$$ 
	where $G_{\vartheta,\alpha}$ is a maximally monotone operator, whose resolvent is $D_{\vartheta,\alpha} = \Id - \alpha\nabla\psi_\theta$.
\end{corollary}

\begin{proof}
	The results follows directly from Proposition~\ref{prp:convpnp}: the uniqueness condition in Proposition~\ref{prp:convpnp} is automatically satisfied, since $\vH^\top\vH+2\Id\succ \bm{0}$ and the linear system has one and only one solution.
\end{proof}

\begin{remark}
	For every $\alpha>0$,
	\[
	\operatorname{Fix}(D_{\vartheta,\alpha})
	=
	\operatorname{zer}(\nabla\psi_\vartheta).
	\]
	If $\psi_\vartheta$ is convex, then
	\[
	\operatorname{zer}(\nabla\psi_\vartheta)
	=
	\argmin{\vx\in\rd^n}\psi_\vartheta(\vx).
	\]
	This establishes a natural connection with RED-PRO-type formulations, where the denoising prior is represented through the fixed-point set of the denoiser. This connection is studied in
	depth in \cite{zhang2026rethinking}, where the convergence analysis
	assumes a proper, convex, lower semicontinuous, differentiable data
	fidelity term with Lipschitz-continuous gradient. The numerical
	experiments considered therein mainly rely on a least-squares data
	fidelity. The standard Poisson KL fidelity does not, in general,
	have a globally Lipschitz-continuous gradient on its natural domain,
	and therefore requires additional care.

\end{remark}

\begin{remark}
	The regularization parameter appearing in \eqref{eq:varpb} is encompassed into the denoiser operator.
\end{remark}

\section{Input Convex Neural Network}
\label{sec:ICNN}
This section shows the architecture and provides the training details for the function $\psi_\vartheta$ and the corresponding denoiser $D_{\vartheta,\alpha}$ used in the numerical experiments of \cref{sec:num}.

\subsection{Architecture and Training}
\label{ssec:training}
The employed network is a modification of the ICNN \cite{pmlr-v70-amos17b} available in the \texttt{DeepInverse} Python package \cite{tachella2025deepinverse}. The networks encompasses 5 convolutional layers, each with 16 filters of dimension $3\times3$ and no bias, with stride one and padding of one pixel. The initialization of the net's weights is uniform for \texttt{wz} layers, while for the other layers the weights are drawn from a normal distribution. Moreover, the net is trained on square patches of 64$\times$ 64 pixels: the numerical experience showed that this option is the best trade--off between performance and training time. The sole modification relies on the activation function: instead of using a classical LeakyReLU function, a custom smoothed version is employed. Such function reads as follows
\begin{equation}
	\label{eq:customAct}
	\rho_{a,c}(t)
	=
	a t
	+
	\frac{1-a}{c}
	\log\left(1+\exp(c t)\right),
	\qquad
	a\in[0,1],\quad c>0.
\end{equation}

In the numerical experiments, the parameters are set to $a=0.1$ and $c=5$. The function $\rho_{a,c}$ is a smooth approximation of the LeakyReLU activation, since $\lim_{c\to\infty}\rho_{a,c}(t) = {\rm LeakyReLU}_a(t)$. Its first and second derivatives are $\rho_{a,c}'(t)
=
a+(1-a)\sigma(c t), \rho_{a,c}''(t)
=
(1-a)c
\sigma(c t)\bigl(1-\sigma(c t)\bigr)$, respectively,  where $\sigma(t)=(1+\exp(-t))^{-1}$. Hence, $\rho_{a,c}$ is nondecreasing and convex.  The smooth activation preserves the convexity requirements of the ICNN while avoiding the nondifferentiability of the standard LeakyReLU.

The initial stage of training is performed on noisy images, perturbed with Gaussian noise with $\sigma=0.05$, aiming to provide a stable initialization before extending it to a range of different noise levels. The loss function is the mean-square error loss between clean and denoised images:
\begin{equation}
	\label{eq:LOSS1}
	\mathcal{L}(\vartheta) = \mathbb{E}_{\vx,\bm{\eta},\sigma}\lq\|D_{\vartheta,1}(\vy)-\vx\|_2^2\rq =  \mathbb{E}_{\vx,\bm{\eta},\sigma}\lq\|\nabla\psi_\vartheta(\vy)-\lp\vy-\vx\rp\|_2^2\rq,
\end{equation}
with 
$$
\vy = \vx+\sigma \bm{\eta}, \quad \bm{\eta}\sim\mathcal{G}(\bm{0},\Id),
$$
where $\mathcal{G}(\bm{0},\Id)$ denotes a multivariate Gaussian distribution of zero mean and identity variance. The network's weights are imposed to be nonnegative during the training: this latter condition together with the custom activation function \cref{eq:customAct} ensures the input convexity of the function. The strong convexity option available in the \texttt{DeepInverse} implementation is set to 0.  

The optimization algorithm used for the training is Adam, with learning rate equal to $10^{-3}$, for 40 epochs and minibatch size of 8. For each image of the minibatch, 4 patches of dimension $64\times64$ have been extracted. The dataset for the training consists of a sub set of the DIV2K dataset \cite{Agustsson_2017_CVPR_Workshops}. To limit the computational cost, only the first 100 training images are used, and from each image several patches are  randomly extracted at every epoch. The images are normalized in $[0,1]$.

\paragraph*{Fine-tuning for blind denoising.}
Once the first stage of training is performed, the dataset with 100 images is used again for further training. Each minibatch is perturbed with a different amount of noise, with $\sigma$ randomly selected (from a uniform distribution) in $[0.01, 0.05]$. The optimization settings are still the same: Adam algorithm is used, with learning rate equal to $10^{-3}$, 40 epochs and minibatch size of 8 images, and from each one 4 patches with dimension $64\times64$ are extracted. The loss remains the one in \eqref{eq:LOSS1}.

\paragraph*{Smoothness-oriented fine-tuning.} The particular training choice made in this work considers the case $\alpha=1$: therefore, a further training is performed, in order to promote the condition $\alpha\leq1/L\Leftrightarrow\alpha\,L\leq1\Leftrightarrow L\leq1$. 

This aim is pursued by employing the power iteration \cite{booth2006power} approach to control the eigenvalues of the Hessian. A further term in the loss function has been added:
$$
R_{\rm FNE}(\theta) = \mu_{\rm FNE}\,{\rm ReLU}(l_\theta(\vy)-L_{\rm target})^2, \quad l_\theta(\vy)=\lambda_{\rm max}(\nabla^2\psi_\vartheta(\vy))
$$
where $L_{\rm target}$ is the upper bound of the Lipschitz constant, and $\mu_{\rm FNE}\,$ a regularization parameter. $l_\theta(\vy)$ is the power iteration estimate of the largest eigenvalue of the Hessian at a sample input $\vy$. In this second fine tuning  $L_{\rm target}=1$ and $\mu_{\rm FNE}\,=0.001$: the complete loss function reads hence as
$$
\mathcal{L}(\vartheta) + R_{\rm FNE}(\vartheta).
$$ 
The number of steps for the power iteration algorithm is set to 5, for saving computational cost and time. The optimizer is Adam, with learning rate equal to 0.001, and the training lasts for 40 epochs. This setting still provides reliable results, as shown in \cref{ssec:valDen}.

\begin{remark}
	This penalty promotes the condition $L_\theta\leq1$ on the sampled training distribution, but it does not provide a global Lipschitz gradient.
\end{remark}

\subsection{Validation of the Denoiser}
\label{ssec:valDen}
The validation of the trained denoiser is performed on 10 images of the test subset of DIV2K dataset. Several noise levels are considered, and for each noise level 100 $256\times256$ patches are extracted, 10 per image. \cref{tab:recover} reports the average relative error on noisy and denoised image, together with the gain, for each level considered. The trained denoiser abides to its task for the whole range considered, \emph{i.e.}, $\sigma\in[0.01,0.05]$, even if the improvement at level 0.01 is limited. It is also able to provide reliable results even for noise levels outside the training range, when $\sigma>0.05$, stabilizing at a ratio between the relative error on the reconstruction and on $\vg$ around 1.8. 
\begin{table}[htp]
	\tbl{Denoising performance of $D_{\vartheta,1}$. The relative error with respect to $\vx^\star$, the clean data, of $\vg$ and of the reconstruction $\tilde\vx$ are shown: these are the averages among 100 different patches, extracted from 10 different images, 10 per image. The Gain column shows the  ratio between the average relative errors: a ratio greater than 1 means an improved reconstruction with respect to the data $\vg$.}{
	\begin{tabular}{c|c|c|c||c|c|c|c}
		\toprule
		$\sigma$ & Rel. Err. $\vg$ & Rel. Err. $\tilde\vx$ & Gain &$\sigma$ & Rel. Err. $\vg$ & Rel. Err. $\tilde\vx$ & Gain\\
		\toprule
		0.01 & 0.0270 & 0.0262 & 1.0317 & 0.06 & 0.1632 & 0.0892 & 1.8297\\
		0.02 & 0.0524 & 0.0346 & 1.5141 & 0.07 & 0.1773 & 0.0958 & 1.8515\\
		0.03 & 0.0843 & 0.0505 & 1.6698 & 0.08 & 0.1992 & 0.1074 & 1.8547\\
		0.04 & 0.1081 & 0.0611 & 1.7678 & 0.09 & 0.2340 & 0.1255 & 1.8643\\
		0.05 & 0.1243 & 0.0695 & 1.7884 & 0.10 & 0.2659 & 0.1427 & 1.8637\\
\end{tabular}}
 \label{tab:recover}
\end{table}

%

The second check  consists in evaluating the value for $L$, the smoothness constant. This analysis is performed using the previous 10 images from the DIV2K dataset. A total of 100 $64\times64$ patches are randomly extracted from these images and corrupted with Gaussian additive noise, whose standard deviation is chosen uniformly in $[0.01,0.05]$. 

For each patch, the estimation of the constant $L$ is reached by power iteration method, estimating the maximum eigenvalue of $\nabla^2\psi_\vartheta$, using 20 iterations and 5 random restarts, therefore in a more accurate manner than in the training phase, where the number of power iterations has been fixed to 5 for saving computational time. 

Over the 100  inputs, the local smoothness estimate $\hat L$ of the Lipschitz constant provides an average value of 1.0662, with a standard deviation of 0.002, and a maximum value of 1.069. A further analysis is carried on out-of-sample 200 patches from images of the BSD500 dataset \cite{5557884}, under the same settings. The average estimate of the smoothness constant is $1.066$, with standard deviation $0.001$, and maximum value $1.07$. All the estimated values are above 1 in both experiments.

The experiments of \cref{sec:num} are actually performed on images and/or patches with dimension $256\times256$: a further analysis is done on patches of such dimension, and the previous results are confirmed. The average local estimate $\hat L$ obtained on 100 samples, under the same above settings for the power iteration, is $1.06$, with a standard deviation of $0.001$ and a maximum achieved value of $1.07$. These validation experiments show that the training phase for forcing $L\leq 1$ was not completely successfully.

Since for the requirement for the denoiser $D_{\vartheta,\alpha}$ to be FNE is $\alpha L\leq 1$ and the training phase has been performed setting $\alpha=1$, this empirical evidence, on both test and out-of-sample images, shows that for the considered validation inputs such denoiser operates in an FNE-compatible regime if the relaxation parameter is chosen according to Lemma~\ref{lem:FNE}, \emph{i.e.} $\alpha\hat L\leq1$.

\section{Numerical Experiments}
\label{sec:num}

This section is devoted to assessing the performance of the proposed
\gsdsplit algorithm. The first part investigates the hyperparameter settings
and the robustness of the algorithm with respect to the relaxation parameter
$\alpha$. The second part focuses on a comparison with the original \pnpsplit.

The training phase was performed on Google Colab because of hardware
limitations, whereas the validation and performance tests were carried out
on a MacBook Pro equipped with an Apple M4 processor, 16 GB of unified memory,
and the MPS backend. The code is available at \url{https://github.com/AleBenfe/PnPSplitPlus}.

The numerical assessment is performed using the Peak to Signal Noise Ration (PSNR) \cite{martini2025measuring}, the Relative Error (RE) with respect to the clean image, and using the Similarity Structure Index Measure (SSIM) \cite{Wang2004}. The linear operator $\vH$ is chosen as a Gaussian blur with standard deviation equal to 1, while the Poisson noise level $\nu$ is set to 20. The number of iterations is set to 400, the initial iterate is set equal to $\vg$, and the relative $\vw^0$ is set accordingly. The relaxation parameter $\alpha$, accordingly to the validation experiments done in \cref{ssec:valDen}, is set to 0.9.
\begin{figure}[htbp]
	\centering
	\newcommand{\factor}{0.3}
	\subfloat[$\vx^0=\vg$]{\includegraphics[width=\factor\textwidth]{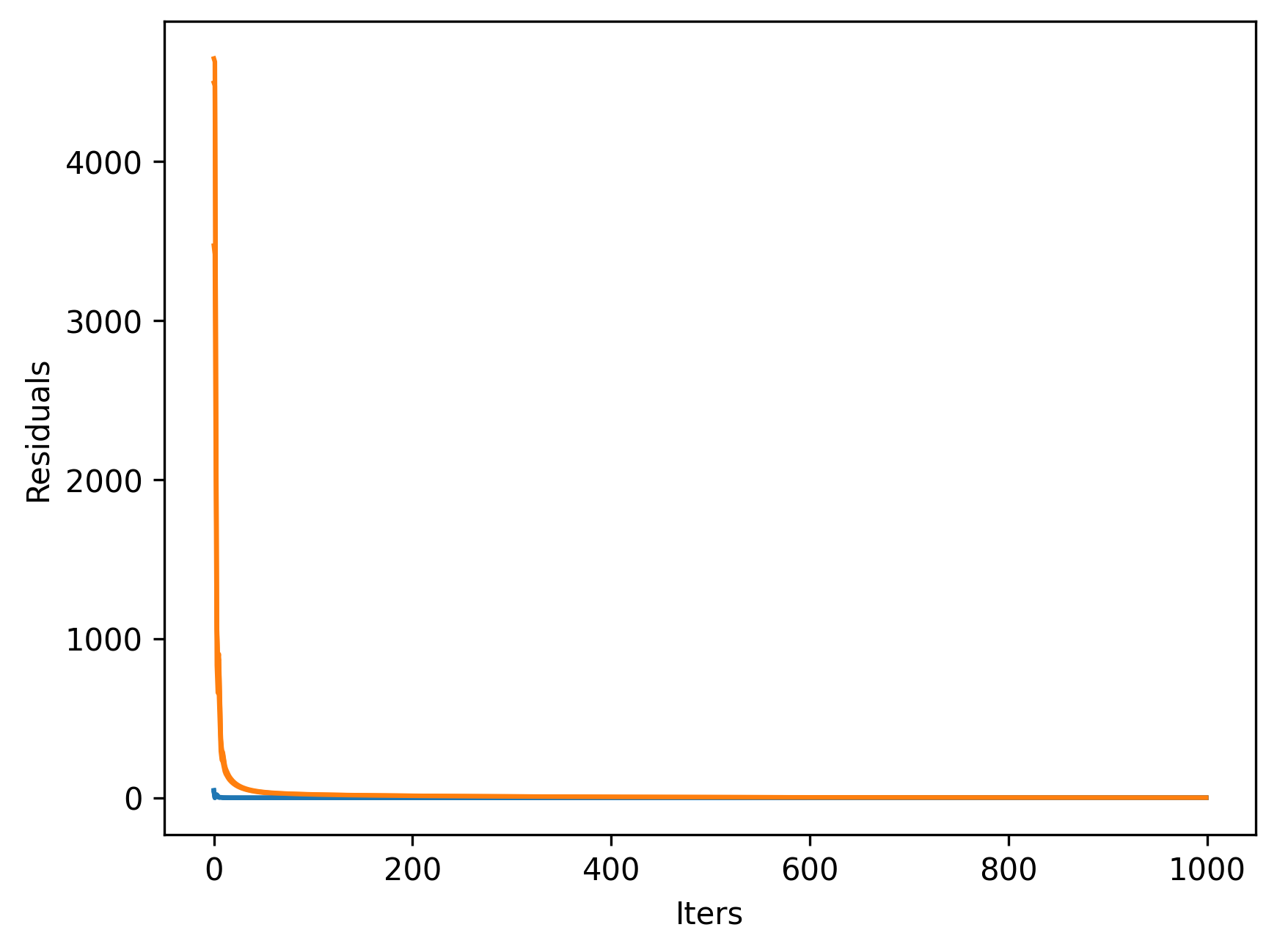}}\hfill\subfloat[$\vx^0=\bm{0}$.]{\includegraphics[width=\factor\textwidth]{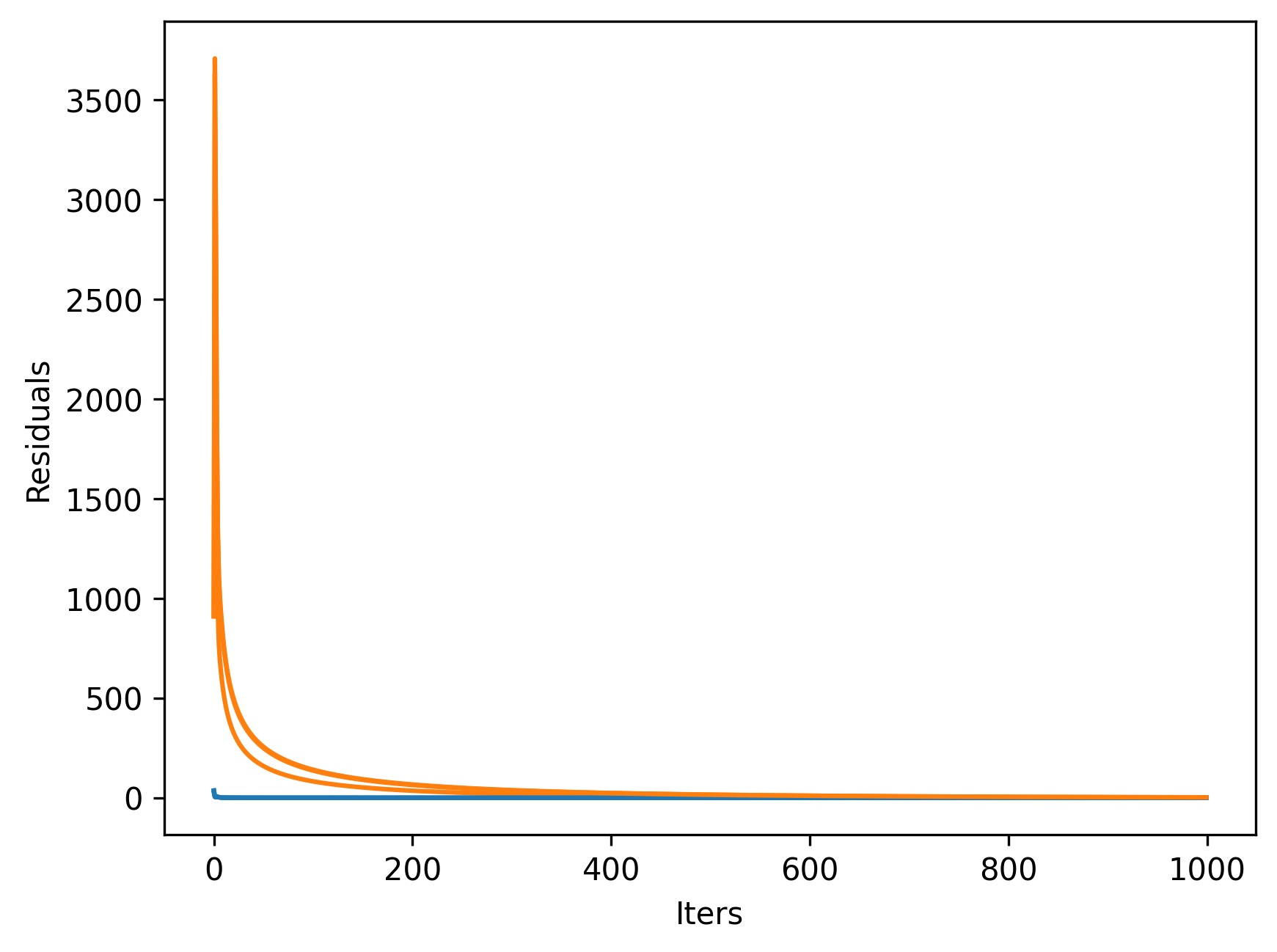}}\hfill\subfloat[$\vx^0$ randomly set.]{\includegraphics[width=\factor\textwidth]{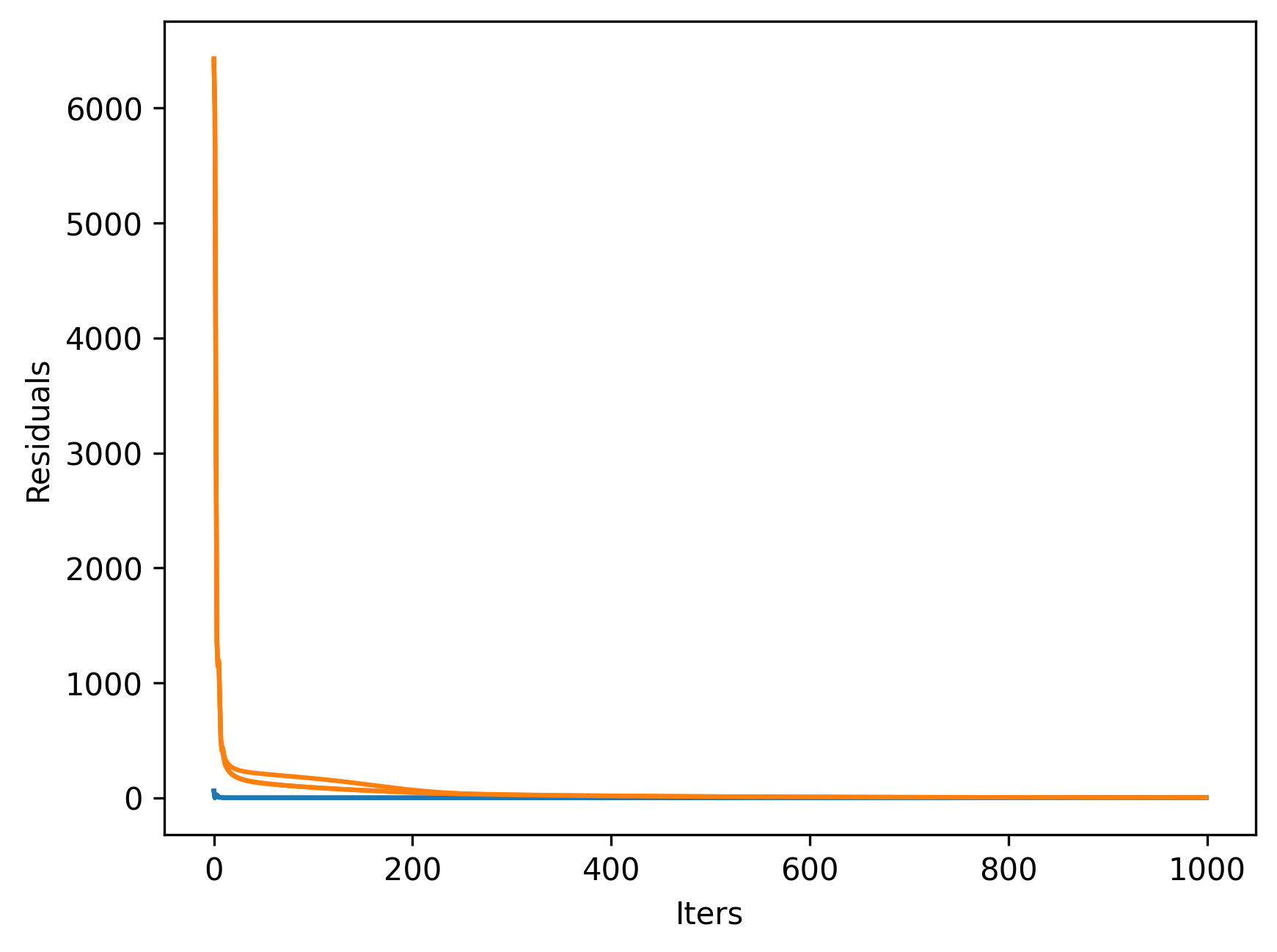}}
	\caption{Convergence assessment. The proposed algorithm  achieves a stationary regime, with both residuals numerically close to 0, when all the parameters are set in order to satisfy a FNE regime for the denoiser. Both primal and dual residual stabilize early among the iterations.\label{fig:conv}}
\end{figure}
The number of iterations is set to 400 since, under the analysis done in \cref{ssec:valDen} and the theoretical results, the denoiser is such that the method converges: \cref{fig:conv} shows the norm of the primal  and dual residuals, namely $r^k, s^k$,  among 1000 iterations:
$$
r^k = \|\vM\vx^k-\vw^k\|_2,\qquad \vs^k = \|\frac{1}{\gamma}\vM^\top(\vwkk-\vwk)\|_2.
$$
The three panels show that even for different initialisations, namely the perturbed data, a zero image and a random image with elements chosen uniformly in $[0,1]$, the algorithm reaches a stationary regime: 400 iterations are sufficient for both residuals to reach the stationary regime observed in the experiment.

\subsection{Hyperparameter Analysis}
\label{ssec:hypAnalysis}
\cref{al:gsdsplit+} depends mainly on two parameters: the ADMM step $\gamma$ and the denoiser's relaxation parameter $\alpha$. The analysis of the algorithm behaviour with respect to these two parameters is performed via two distinct grid searches on 20 $256\times256$ patches extracted from 20 images (one patch per image) of the DIV2K test dataset, which differ from the ones used for the validation in  \cref{ssec:valDen}. 

For finding an (empirical) optimal value for $\gamma$ the first grid search is performed in the interval $[0.01,0.2]$ with a step of $0.01$. For each value for $\gamma$ the average PSNR and its standard deviation are considered across the 20 images. The results are presented in \cref{tab:gammaPSNR}: the value that provides the highest average PSNR on this dataset is 0.03. It is worth to point out that there is a set of values, namely between 0.01 and 0.07, that achieve reliable results, while for larger values the performance is worsening but it is still bounded.
\begin{table}[htbp]
	\tbl{Average PSNR and standard deviation obtained for different values
		of the ADMM parameter $\gamma$, with $\alpha=0.9$. The largest
		average PSNR is attained for $\gamma=0.03$.}
		{\begin{tabular}{c|c||c|c||c|c||c|c}
				\toprule
			$\gamma$ & PSNR& $\gamma$ & PSNR& $\gamma$ & PSNR& $\gamma$ & PSNR\\
			\midrule
0.01 & 22.66 $\pm$ 1.96 & 0.06 & 23.41 $\pm$ 1.58& 0.11 & 22.75 $\pm$ 1.35&0.16 & 22.23 $\pm$ 1.23\\
0.02 & 23.65 $\pm$ 1.98 & 0.07 & 23.26 $\pm$ 1.51& 0.12 & 22.64 $\pm$ 1.32&0.17 & 22.13 $\pm$ 1.22\\
0.03 & 23.77 $\pm$ 1.87 & 0.08 & 23.13 $\pm$ 1.46& 0.13 & 22.53 $\pm$ 1.29&0.18 & 22.04 $\pm$ 1.20\\
0.04 & 23.69 $\pm$ 1.75 & 0.09 & 22.99 $\pm$ 1.41& 0.14 & 22.42 $\pm$ 1.27&0.19 & 21.96 $\pm$ 1.19\\
0.05 & 23.56 $\pm$ 1.66 & 0.10 & 22.86 $\pm$ 1.37& 0.15 & 22.32 $\pm$ 1.25&0.20 & 21.88 $\pm$ 1.17\\
		\end{tabular}}
		\label{tab:gammaPSNR} 
\end{table}
\begin{remark}
	A fixed value of $\gamma$ is adopted throughout this work. Adaptive
	strategies, such as the one employed in \cite{benfenati2025plug}, are left
	for future investigation.
\end{remark}
A sensitivity analysis is performed with respect to the relaxation parameter $\alpha$. The denoiser has been trained for $\alpha=1$, and the validation shown in \cref{ssec:valDen} proved that the denoiser satisfied a local $L$--Lipschitz empirical property, on subsets of different validation sets, with an estimate $\hat L\sim1.06$, with a maximum value of 1.07. Lemma~\ref{lem:FNE} states that as long as $\alpha\leq L^{-1}$, the denoiser is firmly non expansive. In the previous assessment, $\alpha$ has been set to 0.9 for being in a FNE regime. Nonetheless, since PnP methods showed to achieve remarkable results even in absence of convergence properties, the performance of \gsdsplit is evaluated for  $\alpha\in\{0.1,0.2,\dots,1.5\}$: this set encompasses values for $\alpha$ that satisfy the FNE empirical regime and  values that does not ensure it. \cref{tab:alphaPSNR} considers as a merit measure the average PSNR on the same subset of images employed for the grid search for $\gamma$.
\begin{table}[htbp]
	\tbl{Average PSNR and standard deviation obtained for different values
		of the relaxation parameter $\alpha$, with $\gamma=0.03$.}
		{\begin{tabular}{c|c||c|c||c|c}
			$\alpha$ & PSNR& $\alpha$ & PSNR& $\alpha$ & PSNR\\
			\midrule
			0.1 & 21.23 $\pm$ 1.03 & 0.6 & 23.61 $\pm$ 1.71& 1.1 & 23.75 $\pm$ 1.95\\
			0.2 & 22.37 $\pm$ 1.24 & 0.7 & 23.70 $\pm$ 1.78& 1.2 & 23.71 $\pm$ 1.97\\
			0.3 & 22.93 $\pm$ 1.40 & 0.8 & 23.75 $\pm$ 1.84& 1.3 & 23.67 $\pm$ 1.99\\
			0.4 & 23.26 $\pm$ 1.52 & 0.9 & 23.77 $\pm$ 1.88& 1.4 & 23.60 $\pm$ 2.00\\
			0.5 & 23.47 $\pm$ 1.63 & 1.0 & 23.77 $\pm$ 1.92& 1.5 & 23.53 $\pm$ 2.00\\
		\end{tabular}}
		\label{tab:alphaPSNR}
\end{table}  
\Cref{tab:alphaPSNR} shows that values around $\alpha=0.9$ provides better average reconstruction. It is evident, anyway, that even values outside the empirically supported FNE-compatible range, namely $\alpha=1$ and larger, provides reliable reconstructions, on average. The results show a broad performance plateau around $\alpha=0.9$. In particular, the theoretically motivated choice $\alpha=0.9$ achieves one of the largest average PSNR values. Values $\alpha\geq1$, which lie outside the empirically supported range, still produce numerically stable reconstructions in this experiment. However, the sufficient condition used in the convergence analysis is not supported for these values, and no convergence claim is made in this case.

As a further check, the quantity 
$$
m_k(\alpha)
=
\langle \Delta z^k,\Delta q^k\rangle
-
\alpha\|\Delta q^k\|_2^2
$$
is computed, where 
$$
z^k = x^{k+1}+\lambda_2^k,
\qquad
\Delta z^k=z^k-z^{k-1}, \quad 
\Delta q^k
=
\nabla\psi_\vartheta(z^k)
-
\nabla\psi_\vartheta(z^{k-1}).
$$ 
Since $D_{\vartheta,\alpha}=\Id-\alpha\nabla\psi_\vartheta$, the condition $m_k(\alpha)\geq0$ is equivalent to the FNE inequality in Definition~\ref{def:FNE} on the considered pair. \cref{fig:checkFNE} depicts four different choices for $\alpha$: for $\alpha=0.1$ and $\alpha=0.9$, the sampled quantities remain nonnegative among the iterations. For $\alpha=1.2$ and $\alpha=1.5$, negative $m_k$ is observed during the first iterations, and  it becomes nonnegative at later iterations. This experiment only assesses the FNE inequality on consecutive inputs encountered along the iterations and does not constitute a global firm-nonexpansiveness property.
\begin{figure}[htbp]
	\centering
	\newcommand{\factor}{0.22}
	\subfloat[$\alpha=0.1$]{\includegraphics[width=\factor\textwidth]{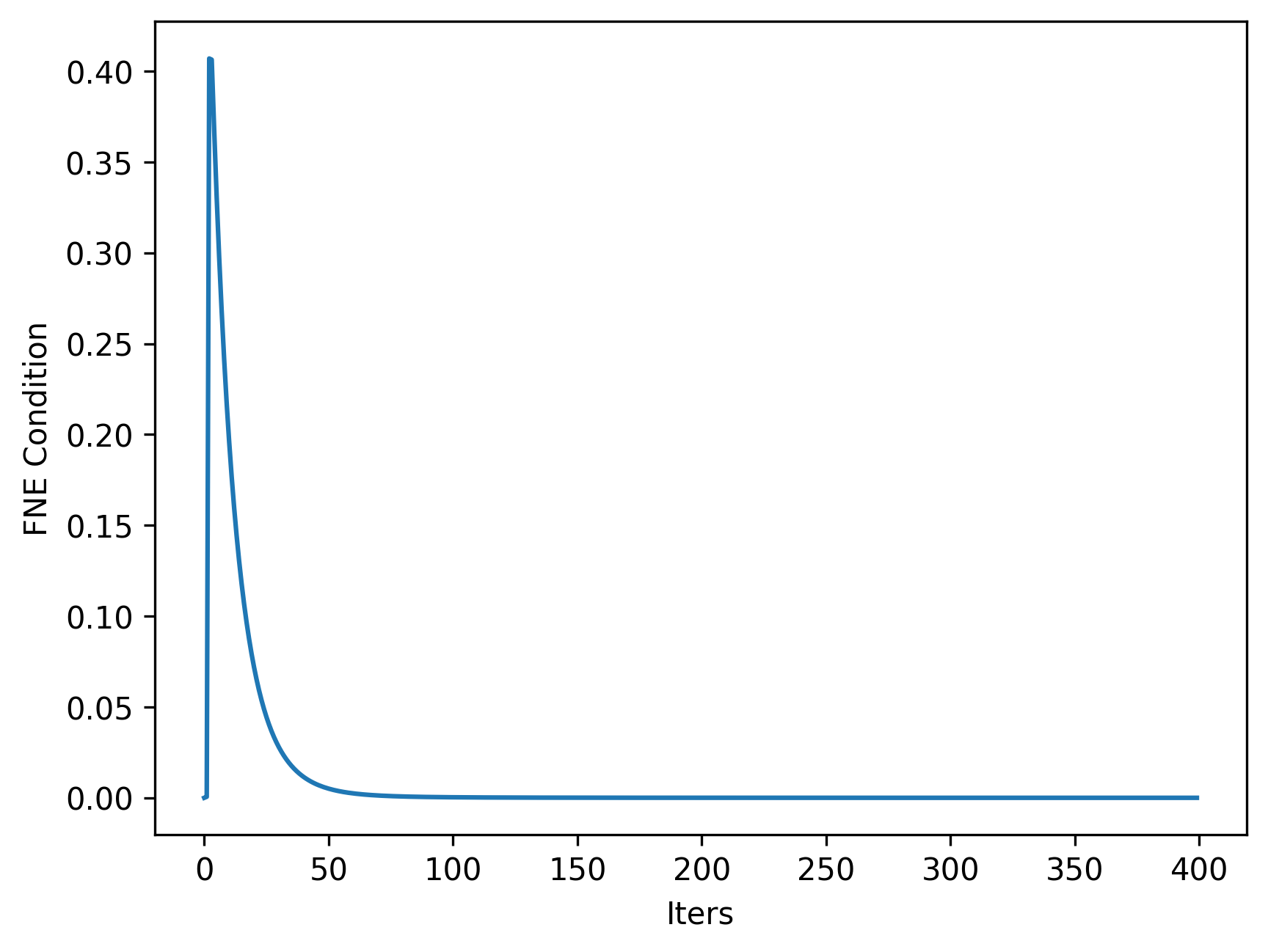}}\hfill
	\subfloat[$\alpha=0.9$]{\includegraphics[width=\factor\textwidth]{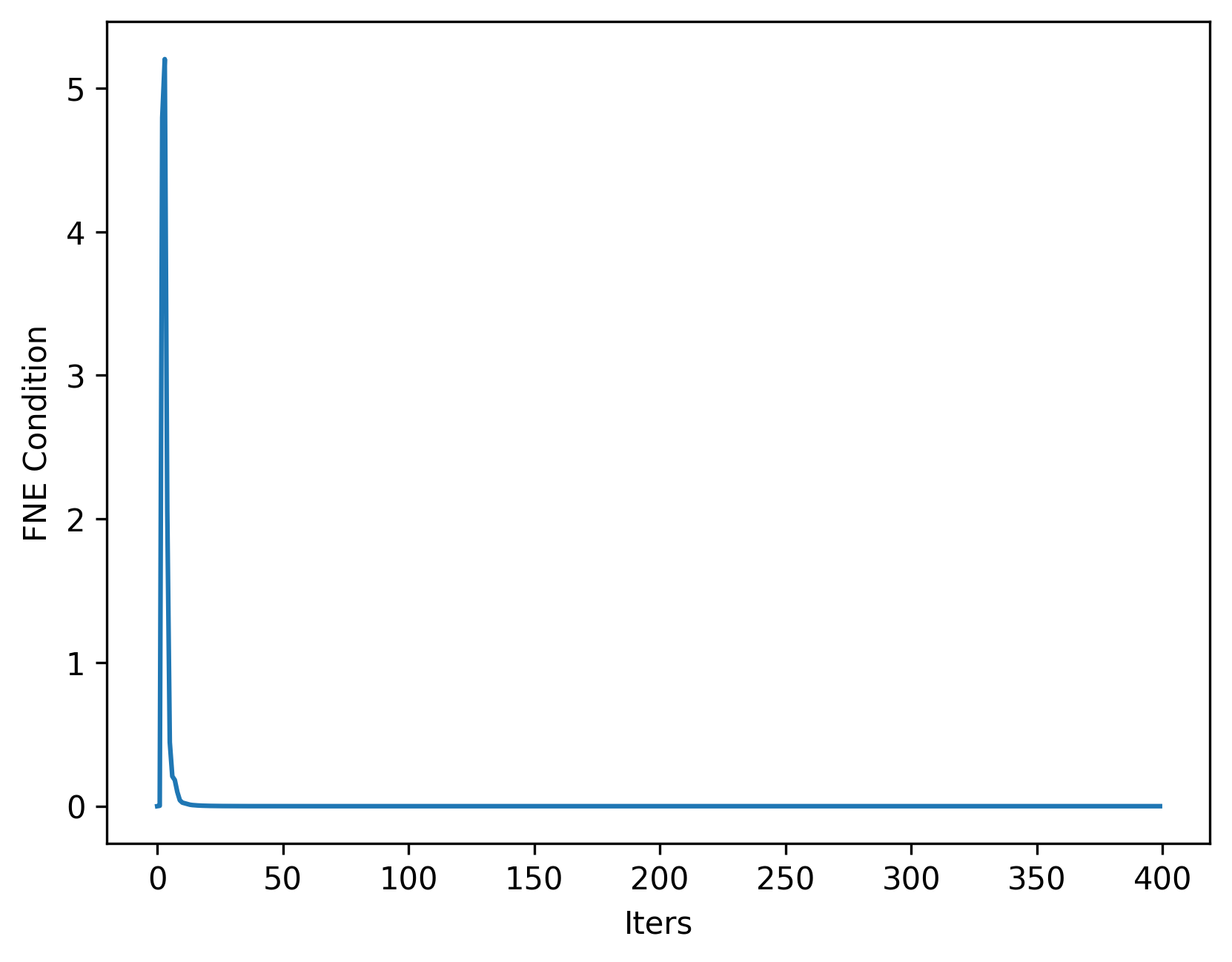}}\hfill
	\subfloat[$\alpha=1.2$]{\includegraphics[width=\factor\textwidth]{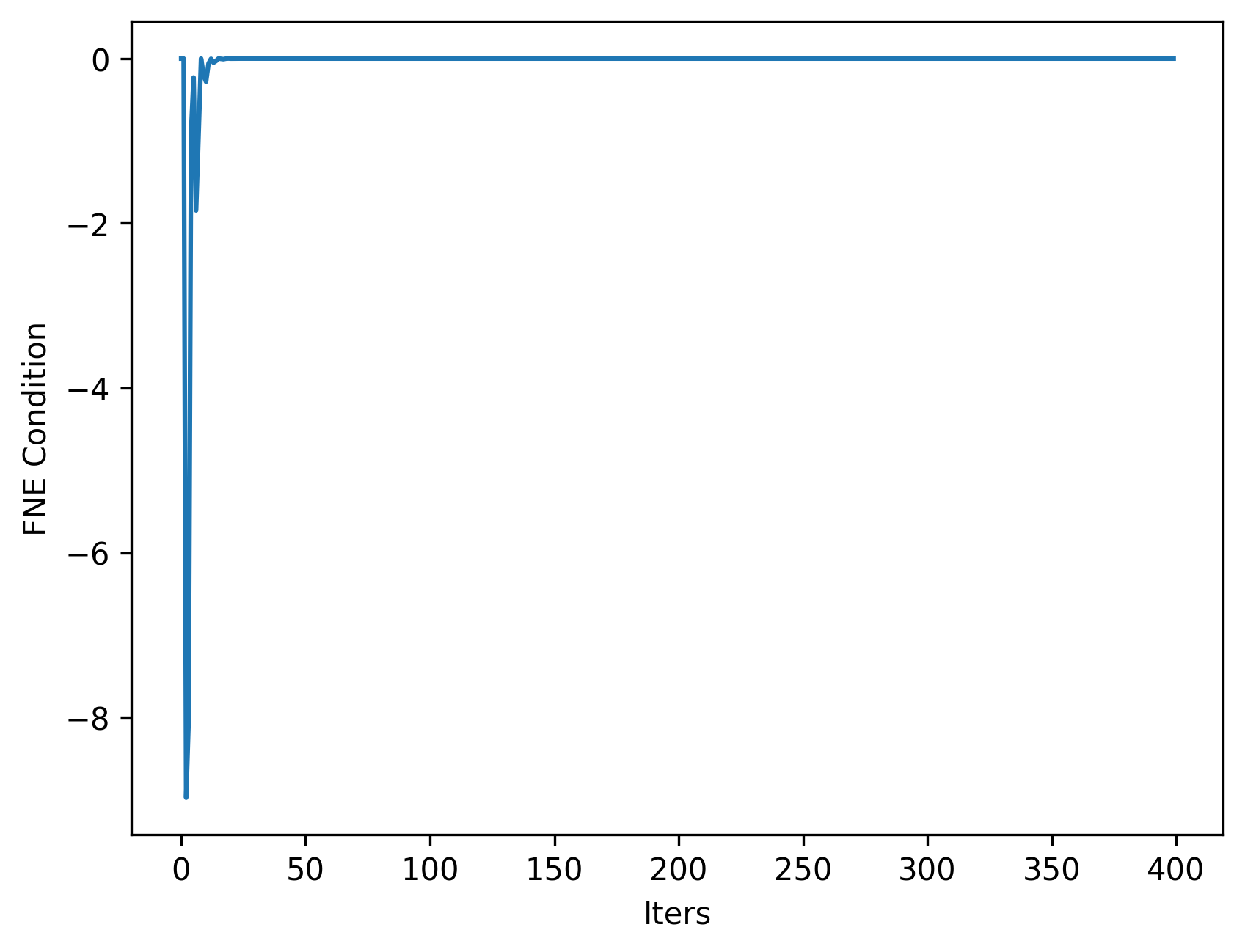}}\hfill
	\subfloat[$\alpha=1.5$]{\includegraphics[width=\factor\textwidth]{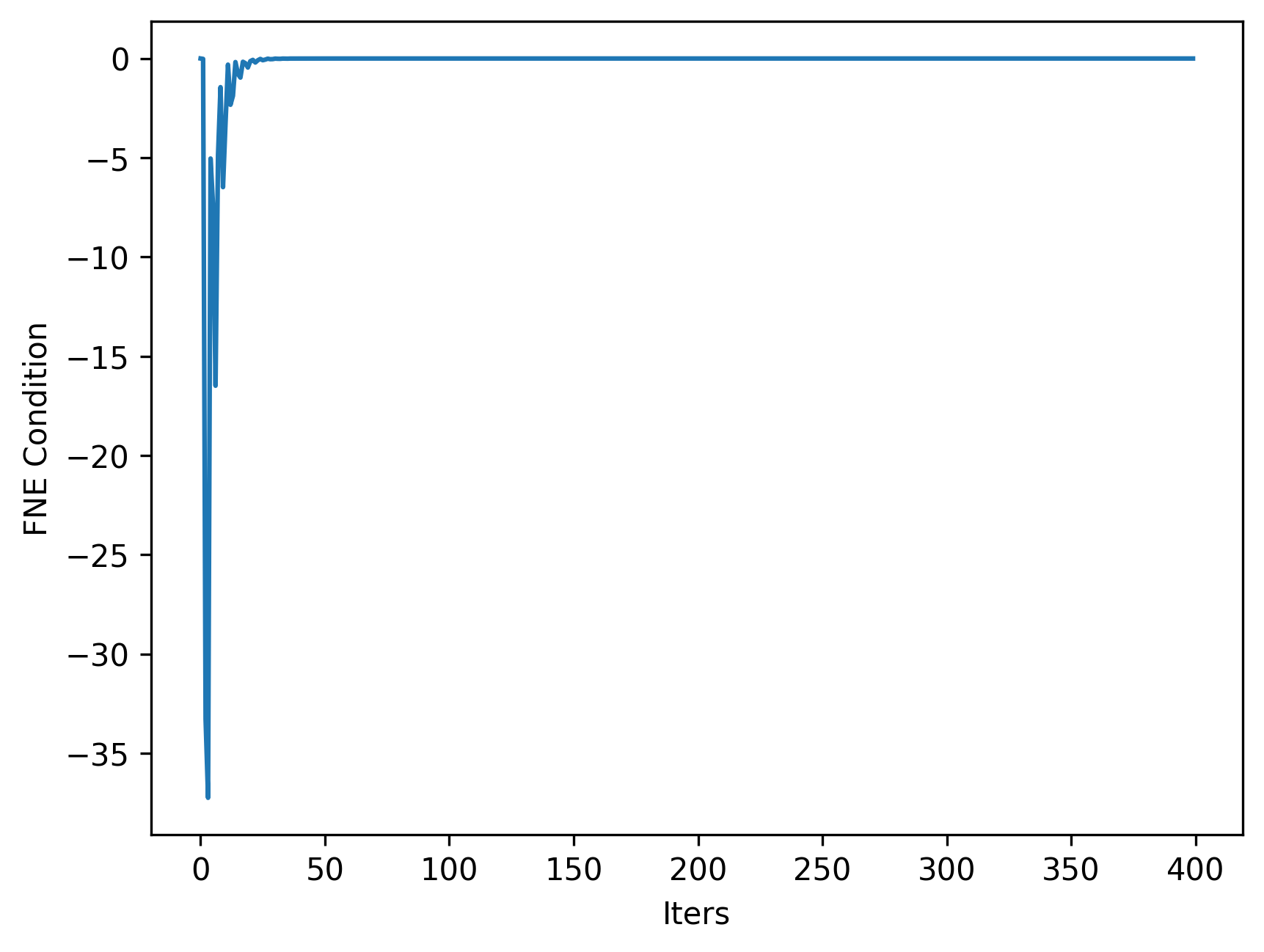}}
	\caption{Check on firmly nonexpansiveness of the denoiser for varying values of $\alpha$. Even when the relaxation parameter is chosen outside the range that theoretically ensures this property, after some iteration the condition is nonetheless met. \label{fig:checkFNE}}
\end{figure}
\subsection{Comparison}
This section compares the proposed \cref{al:gsdsplit+} with the \pnpsplit \cite{benfenati2025plug}, from which stems out. The comparison is performed under this settings:
\begin{itemize}
	\item PnpSplit+: the inner denoiser is the CNN presented in \cite{doi:10.1137/20M1387961}, $\gamma$ is fixed among the iterations at 0.01. 
	\item \gsdsplit: the inner denoiser is based on the ICNN structure introduced in \cref{ssec:training}, $\gamma$ is set to 0.03 and $\alpha=0.9$.
\end{itemize}
For both methods the maximum number of iterations is set to 400. The comparison is done with the optimal value for $\gamma$: for \pnpsplit it is set to 0.01 since it is related to the standard deviation of the denoiser, while for \gsdsplit is the optimal value emerged from the numerical test of \cref{ssec:hypAnalysis}.

The left panel of \cref{fig:plotBestGamma} shows the performances in terms of PSNR of the algorithms with respect to 10 $256\times 256$ patches taken from images belonging to the BSD500 dataset. The right panel of \cref{fig:plotBestGamma} presents, instead, the performances of the algorithms on the same set of images when the values of $\gamma$ are overestimated. For each image, 3 noise simulations are run and the results are collected: the plots in \cref{fig:plotBestGamma} presents the average PSNR for each image and the faint, small shadow represents the standard deviation among the 3 different runs.
\begin{figure}[htbp]
	\centering
	\subfloat[Optimal $\gamma$s.]{\includegraphics[width=0.45\textwidth]{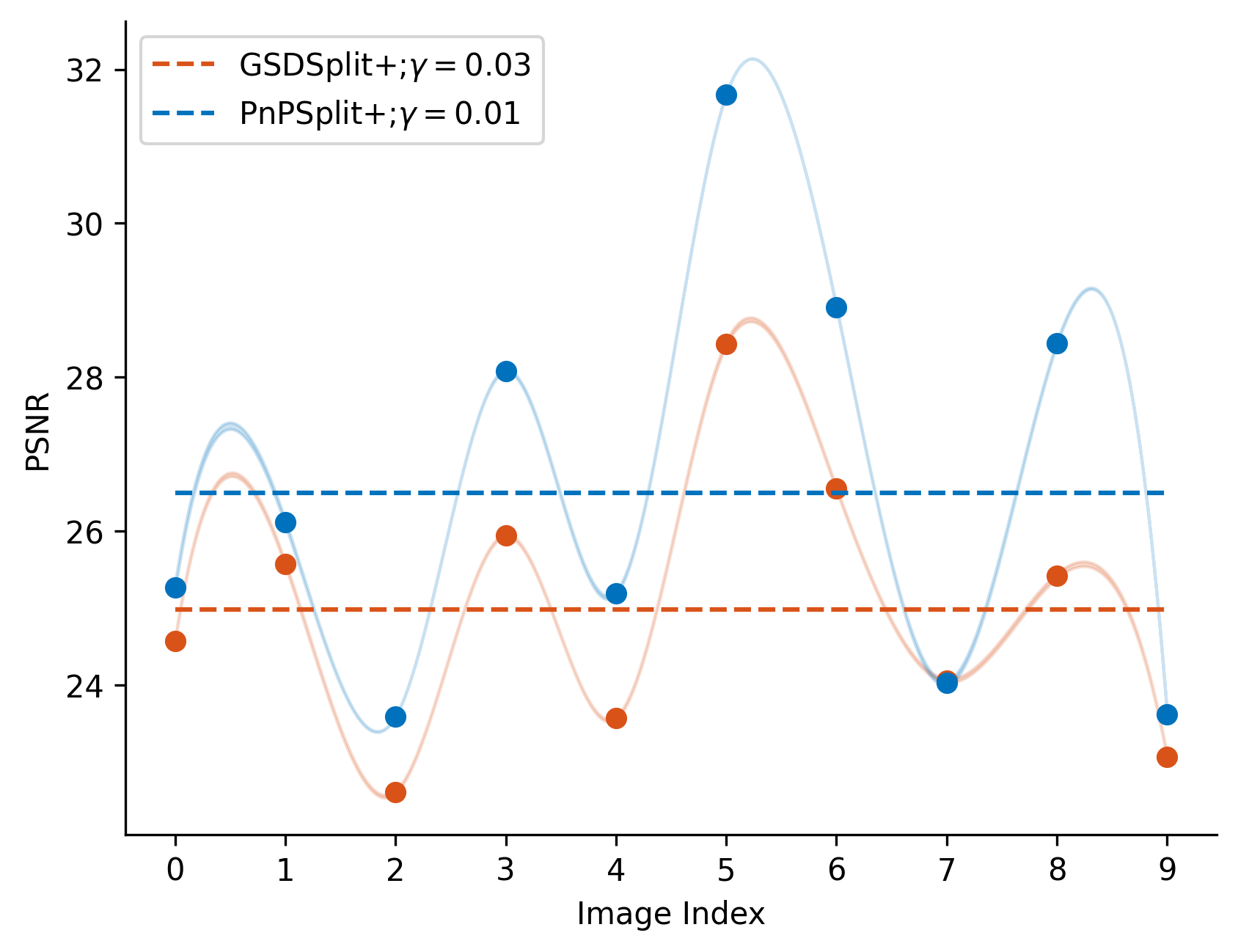}}\hfill
	\subfloat[Overestimated $\gamma$s.]{\includegraphics[width=0.45\textwidth]{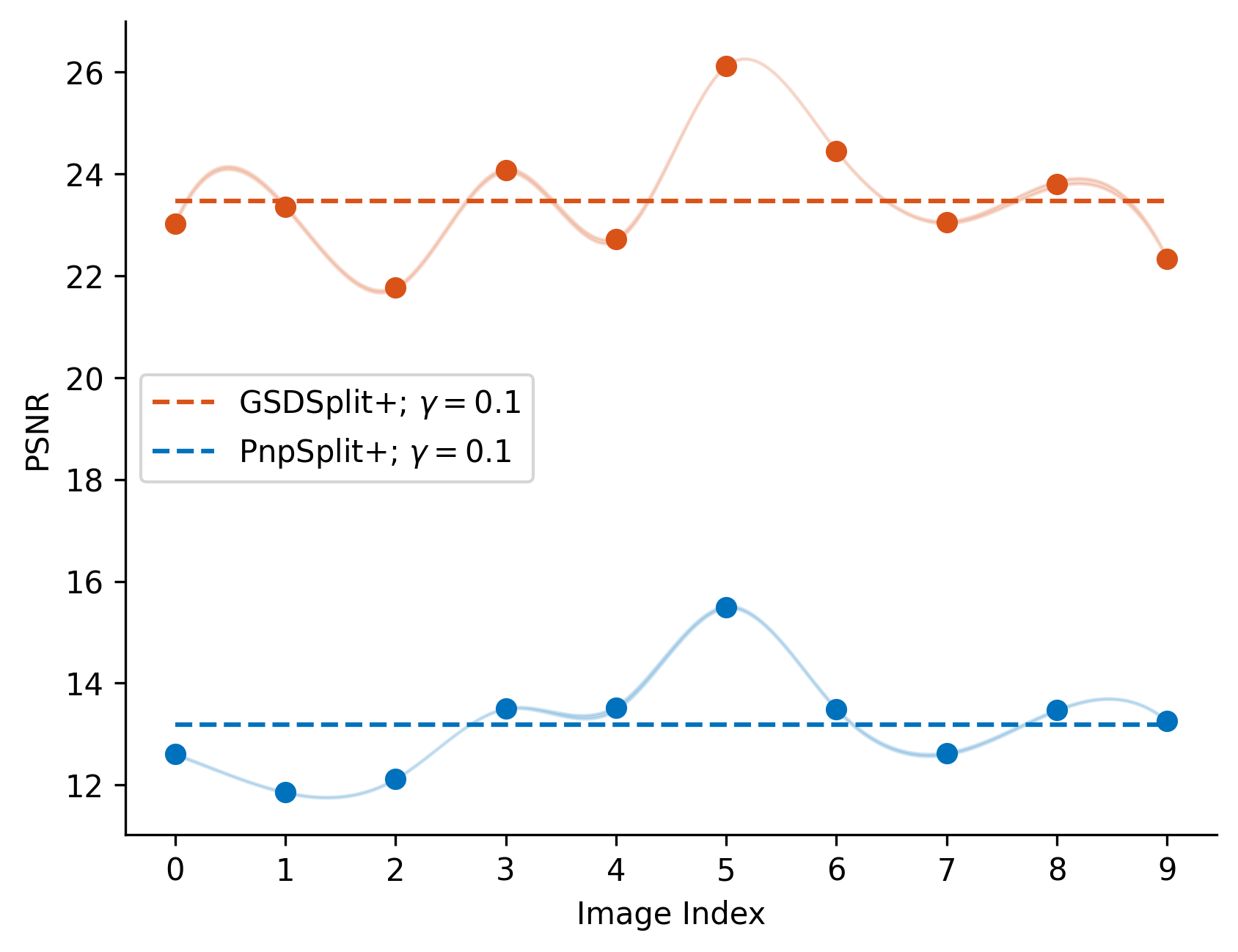}}
	\caption{Left Panel: PSNR of \gsdsplit (orange) and \pnpsplit (blue) when both algorithms are set with the optimal $\gamma$. Right Panel: PSNR of \gsdsplit (orange) and \pnpsplit (blue) when $\gamma$ is overestimated. \label{fig:plotBestGamma}}
\end{figure}
When each method is run with its own optimal value for $\gamma$, \pnpsplit  provides better performances in terms of PSNR. However, the right panel of \cref{fig:plotBestGamma} shows that \gsdsplit suffers a smaller loss, in terms of PSNR, when $\gamma$ is suboptimal. \cref{tab:indexes} shows the comparison of the two algorithms with respect to PSNR, SSIM and relative error.
\begin{table}
	\tbl{Average reconstruction quality. For each image, the metric is first averaged over three independent Poisson
		noise realizations. The reported standard deviation is computed across the
		10 image-wise averages. The first two rows use the optimal
		values of $\gamma$, whereas the last two rows use the overestimated ones.}
	{\begin{tabular}{l|c|ccc}
			\toprule
		Method & $\gamma$ & PSNR $\uparrow$ & SSIM $\uparrow$ & RE $\downarrow$\\
		\toprule
		\pnpsplit & $0.01$ & $26.48\pm2.55$ & $0.69\pm0.12$ & $0.12\pm0.02$\\
		\gsdsplit & $0.03$ & $24.97\pm1.67$ & $0.55\pm0.06$ & $0.14\pm0.02$\\
		\midrule
		\pnpsplit & $0.10$ & $13.19\pm0.96$ & $0.12\pm0.05$ & $0.53\pm0.07$\\
		\gsdsplit & $0.10$ & $23.47\pm1.16$ & $0.45\pm0.09$ & $0.16\pm0.01$
	\end{tabular}}
	\label{tab:indexes}
\end{table}
When each method uses its optimal value of $\gamma$, \pnpsplit achieves a higher average PSNR and SSIM, while the two methods attain comparable relative errors. In particular, the average PSNR difference is approximately $1.51$.

The experiment with overestimated values for $\gamma$ reveals an evident different sensitivity to such parameter. The average PSNR of \pnpsplit decreases by approximately $13.29$, whereas the decrease observed for \gsdsplit is approximately $1.50$. SSIM and relative error have the same behaviour: \pnpsplit performances decreases, whilst \gsdsplit instead suffer less from such choice. On the considered test set, \gsdsplit is substantially less sensitive to a suboptimal ADMM parameter.

This behaviour is confirmed by the plots on \cref{fig:gammas}, where PSNR, SSIM, and relative error are evaluated over the full grid $\gamma\in\{0.01,0.02,\ldots,0.20\}$. On the considered image set, \gsdsplit has a less pronounced decreased behaviour, and it shows less sensitivity to the ADMM parameter. 
\begin{figure}[htbp]
	\centering
	\subfloat{\includegraphics[width=0.31\textwidth]{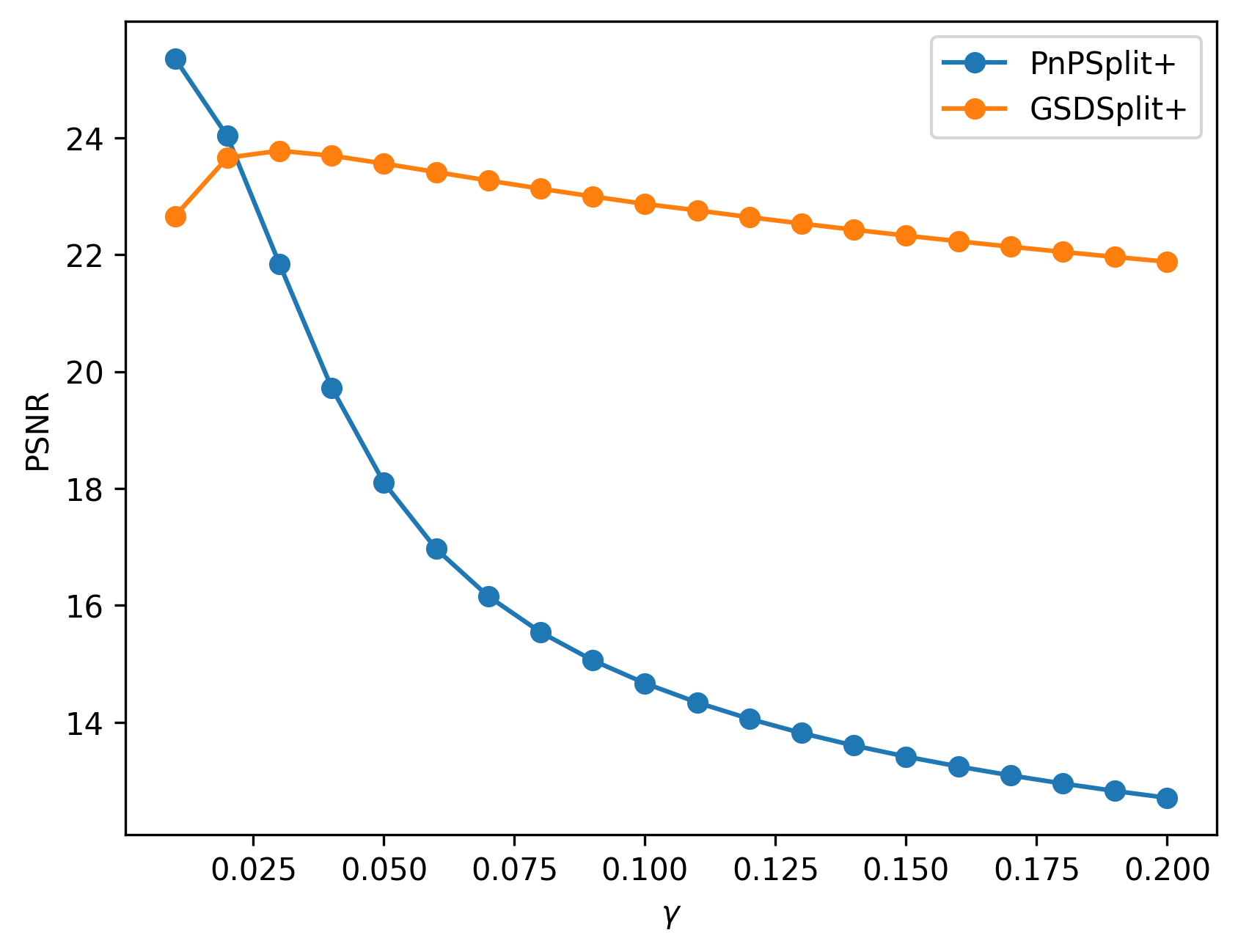}}\hfill	\subfloat{\includegraphics[width=0.31\textwidth]{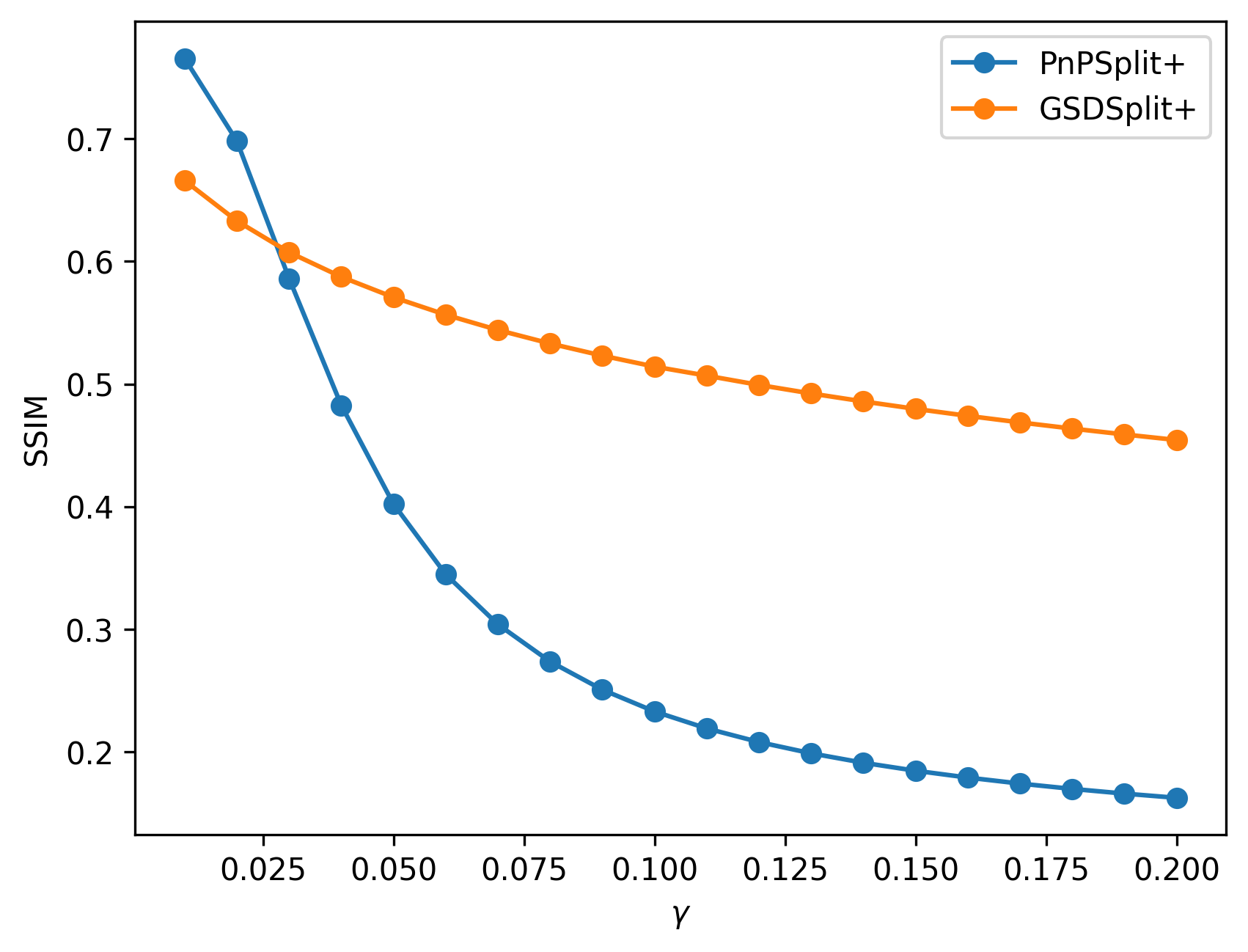}}\hfill	\subfloat{\includegraphics[width=0.31\textwidth]{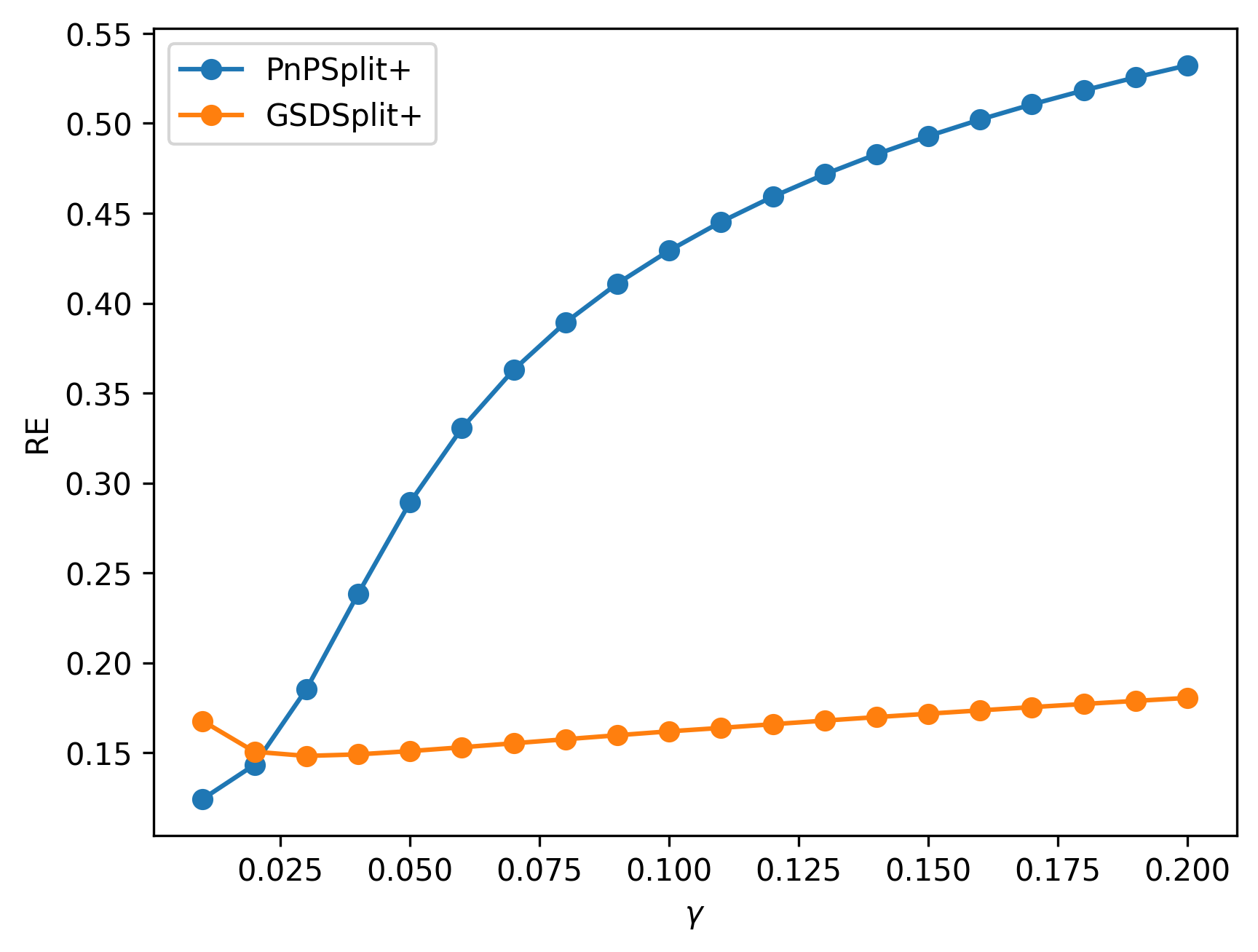}}
	\caption{Left panel: average PSNR. Central panel: average SSIM. Right panel: average RE. Each point of the plot refers to the average of the measure index on the 10 images. The GSD approach suffers less from using a suboptimal value for $\gamma$.\label{fig:gammas}}
\end{figure}
\cref{fig:vis} presents the visual inspections of several reconstruction related to the above experiments: the only difference is that the experiments are performed on the whole image and not only on a smaller patch. When both algorithms are set with the optimal parameters, the reconstruction is reliable, considering the severity of th degradation due to the blur operator and to the high level of noise. The first row of \cref{fig:vis} shows the comparison of the reconstruction when $\gamma$ is set for \gsdsplit and not for PnpSplit+, while the second row shows the opposite case. 
\begin{figure}[htbp]
	\newcommand{\factor}{0.22}
	\centering
\subfloat[$\vx^\star$]{\begin{tikzpicture}[spy using outlines={rectangle, magnification=3, size=0.07\textwidth, connect spies}]
		\node {\includegraphics[width=\factor\textwidth]{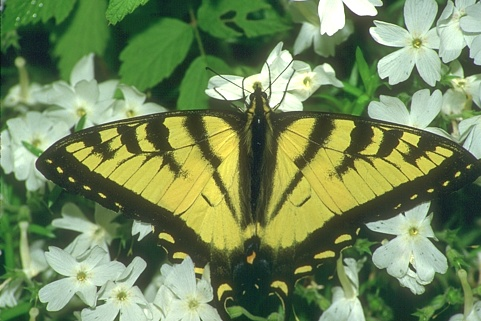}};
		\spy [white, thick] on (-0.5,0.3) in node  at (0.06\textwidth,0.45);
\end{tikzpicture}}\hfill\subfloat[$\vg$]{\begin{tikzpicture}[spy using outlines={rectangle, magnification=3, size=0.07\textwidth, connect spies}]
		\node {\includegraphics[width=\factor\textwidth]{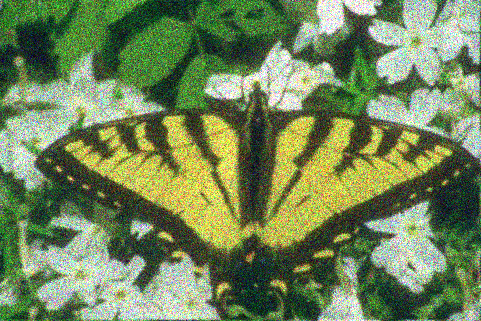}};
		\spy [white, thick] on (-0.5,0.3) in node  at (0.06\textwidth,0.45);
\end{tikzpicture}}\hfill\subfloat[\gsdsplit, $\gamma:0.03$]{\begin{tikzpicture}[spy using outlines={rectangle, magnification=3, size=0.07\textwidth, connect spies}]
		\node {\includegraphics[width=\factor\textwidth]{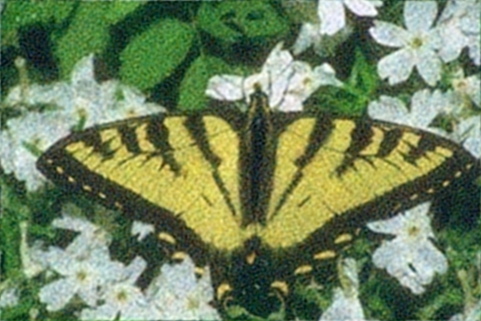}};
		\spy [white, thick] on (-0.5,0.3) in node  at (0.06\textwidth,0.45);
\end{tikzpicture}}\subfloat[\pnpsplit, $\gamma:0.10$]{\begin{tikzpicture}[spy using outlines={rectangle, magnification=3, size=0.07\textwidth, connect spies}]
		\node {\includegraphics[width=\factor\textwidth]{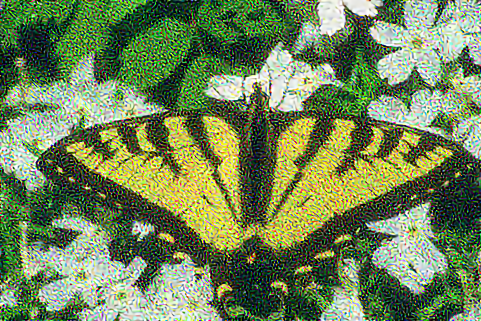}};
		\spy [white, thick] on (-0.5,0.3) in node  at (0.06\textwidth,0.45);
\end{tikzpicture}}

\subfloat[$\vx^\star$]{\begin{tikzpicture}[spy using outlines={rectangle, magnification=3, size=0.07\textwidth, connect spies}]
		\node {\includegraphics[width=\factor\textwidth]{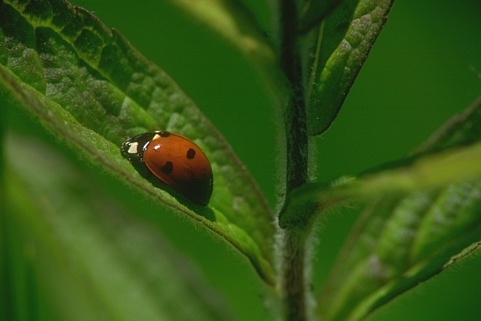}};
		\spy [white, thick] on (-0.5,-0.0) in node  at (0.06\textwidth,0.45);
\end{tikzpicture}}\hfill\subfloat[$\vg$]{\begin{tikzpicture}[spy using outlines={rectangle, magnification=3, size=0.07\textwidth, connect spies}]
		\node {\includegraphics[width=\factor\textwidth]{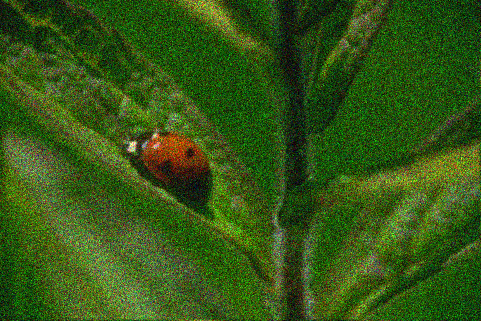}};
		\spy [white, thick] on (-0.5,0) in node  at (0.06\textwidth,0.45);
\end{tikzpicture}}\hfill\subfloat[\gsdsplit, $\gamma:0.10$]{\begin{tikzpicture}[spy using outlines={rectangle, magnification=3, size=0.07\textwidth, connect spies}]
		\node {\includegraphics[width=\factor\textwidth]{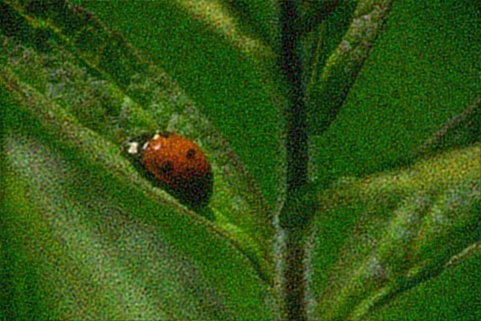}};
		\spy [white, thick] on (-0.5,0) in node  at (0.06\textwidth,0.45);
\end{tikzpicture}}\subfloat[\pnpsplit, $\gamma:0.01$]{\begin{tikzpicture}[spy using outlines={rectangle, magnification=3, size=0.07\textwidth, connect spies}]
		\node {\includegraphics[width=\factor\textwidth]{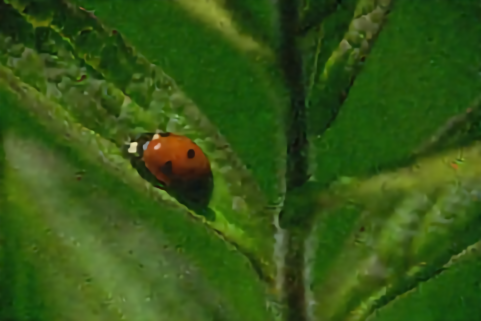}};
		\spy [white, thick] on (-0.5,0) in node  at (0.06\textwidth,0.45);
\end{tikzpicture}}
	\caption{Visual inspection of the reconstructions provided by \gsdsplit and \pnpsplit. The first column depicts the clean image, the second one the perturbed data, the third and the fourth columns show the \gsdsplit and \pnpsplit reconstructions.  \label{fig:vis}}
\end{figure}

\subsection{Computational time and cost}

The computational cost for the creation of the denoiser consists in the training phases. In this work, the strategy relied on 3 different steps of training: firstly, on only on noise level, then on levels randomly drawn from $[0.01,0.05]$ and the last step aimed to force the Lipschitz condition on the function $\psi_\vartheta$. 

The main computational cost for the inference, instead, is related to the application of the denoiser $D_\vartheta$: it amounts to compute the gradient of a convex function, performed with \texttt{autograd} routine in pytorch. The deblurring step involving the inversion of $\vH^\top\vH + 2\Id$, due to the hypothesis on $\vH$ and to the fact it is a blurring operator, can be done via FFT. 

The total computational time for a image reconstruction task depends mainly on the dimension of the image to recover. \cref{tab:compTime} shows these time measurements on the machine used for the experiments. Each run has been performed 5 times, the lowest and the highest time are discarded and the average is considered. As expected, \gsdsplit is outperformed by \pnpsplit: the latter employs a deep CNN, performing several hundreds of convolution operations, but the former has to calculate at each iteration the gradient of a ICNN. There is no free lunch: the robustness to the ADMM parameter $\gamma$ is counterbalanced by the higher computational cost. 
\begin{table}[htbp]
	\tbl{Computational time (in seconds) for different patch sizes and different number of iterations. The last line shows the average computational time per iteration. The patch dimension it is the main contribution to the running time. Each time is obtained by averaging 3 runs, chosen among 5 runs where the smallest and the largest measurements were discarded.}
	{\begin{tabular}{l|rrr||rrr}
			\toprule
			&\multicolumn{3}{c||}{\gsdsplit} & \multicolumn{3}{c}{\pnpsplit}\\
			\cmidrule{2-7}
			& \multicolumn{6}{c}{Patch Dimension}\\
			
			N. It	& 64 	& 128 	& 256 	& 64 	& 128 	& 256 \\ 
			100		& 2.0	& 2.0 	& 3.0	& $<$1.0  & $<$1.0	& 1.3 \\
			400 	& 8.0	& 8.0	& 15.0	& 2.0 	& 3.0	& 7.0\\
			1000 	& 19.6	& 22.3	& 38.3	& 6.0 	& 8.0	& 16.0\\
			\midrule
			s/iter 	& 0.019 & 0.025 & 0.035  & 0.004 & 0.005 &0.016 \\
	\end{tabular}}
	\label{tab:compTime}
\end{table} 

\subsection{Tests on different blur operators and noise levels}

This section is devoted to assess the performance in different scenarios, encompassing other type of linear degradation operators such as motion blur and out-of-focus. Another test is done on an image corrupted by an increasing level of Poisson noise.

\cref{fig:blurs} shows the results achieved when the operator $\vH$ is different from the Gaussian one employed in previous sections and the level of Poisson noise is high. The first row of \cref{fig:blurs} presents an image perturbed by a the same Gaussian blur of the previous experiment, but when the Poisson noise level is set to 5. In the second row the image is perturbed by linear motion blur, with a length of 21 pixels and a slope of 35 degrees. The third row uses an out-of-focus blur with radius 5. The perturbed data $\vg$ in the last two rows are affected by Poisson noise with level equal to 20, as in the experiments of previous sections. 

All the experiments are performed setting $\gamma=0.03$, $\alpha=0.9$ and 400 iterations. \cref{fig:blurs} shows that the proposed method \gsdsplit provide reliable results in different regimes, with several kind of linear blur operators and under high level of Poisson noise. 
\begin{figure}[htbp]
	\newcommand{\factor}{0.3}
	\centering
	
	\subfloat[$\vx^\star$]{\begin{tikzpicture}[spy using outlines={rectangle, magnification=3, size=0.1\textwidth, connect spies}]
			\node {\includegraphics[width=\factor\textwidth]{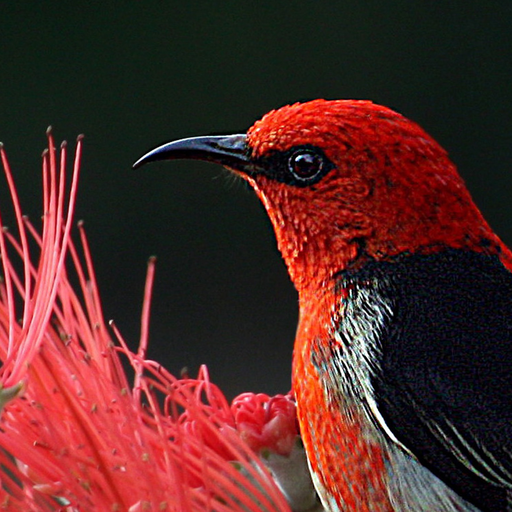}};
			\spy [white, thick] on (0.4,0.8) in node  at (-0.095\textwidth,1.35);
	\end{tikzpicture}}\hfill\subfloat[$\vg$; High Noise.]{\begin{tikzpicture}[spy using outlines={rectangle, magnification=3, size=0.1\textwidth, connect spies}]
			\node {\includegraphics[width=\factor\textwidth]{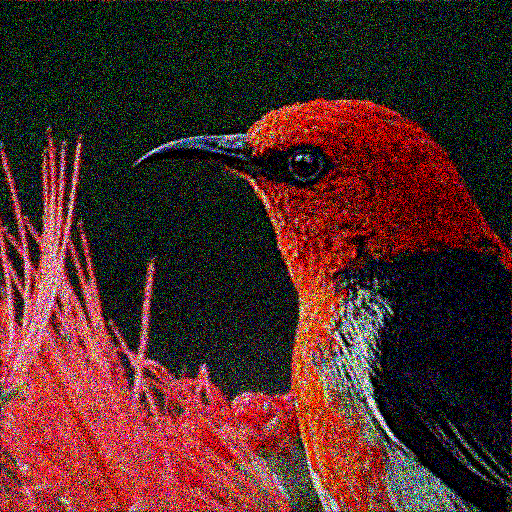}};
			\spy [white, thick] on (0.4,0.8) in node  at (-0.095\textwidth,1.35);
	\end{tikzpicture}}\hfill\subfloat[PSNR:21.58]{\begin{tikzpicture}[spy using outlines={rectangle, magnification=3, size=0.1\textwidth, connect spies}]
			\node {\includegraphics[width=\factor\textwidth]{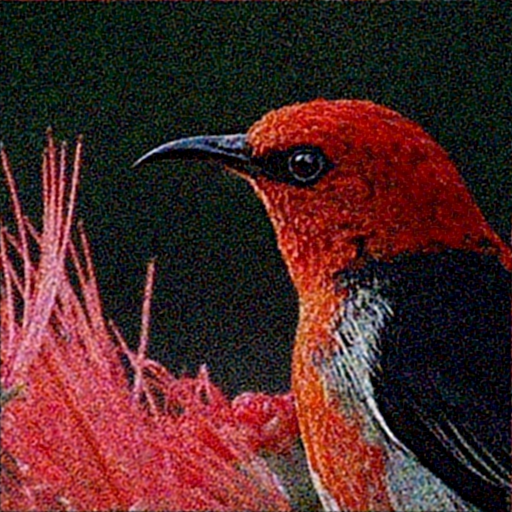}};
			\spy [white, thick] on (0.4,0.8) in node  at (-0.095\textwidth,1.35);\end{tikzpicture}}
	
	\subfloat[$\vx^\star$]{\begin{tikzpicture}[spy using outlines={rectangle, magnification=3, size=0.1\textwidth, connect spies}]
			\node {\includegraphics[width=\factor\textwidth]{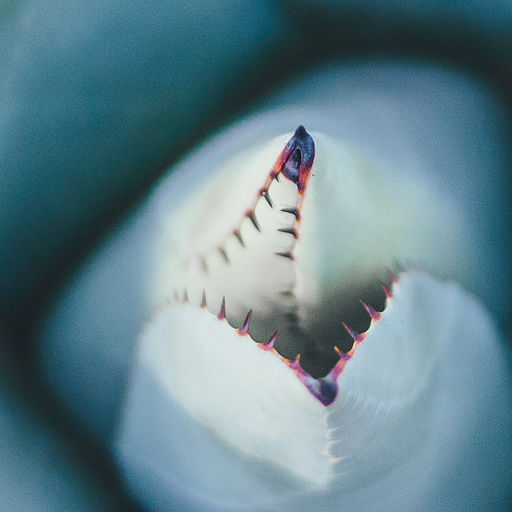}};
			\spy [black, thick] on (0.3,0.3) in node  at (-0.095\textwidth,1.35);
	\end{tikzpicture}}\hfill\subfloat[$\vg$; Motion Blur]{\begin{tikzpicture}[spy using outlines={rectangle, magnification=3, size=0.1\textwidth, connect spies}]
			\node {\includegraphics[width=\factor\textwidth]{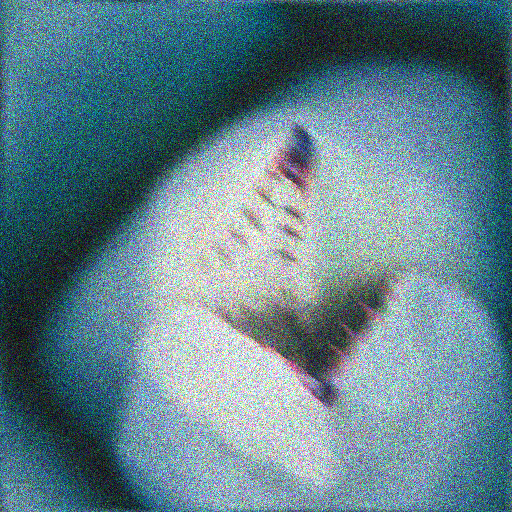}};
			\spy [black, thick] on (0.3,0.3) in node  at (-0.095\textwidth,1.35);
	\end{tikzpicture}}\hfill\subfloat[PSNR:26.34]{\begin{tikzpicture}[spy using outlines={rectangle, magnification=3, size=0.1\textwidth, connect spies}]
			\node {\includegraphics[width=\factor\textwidth]{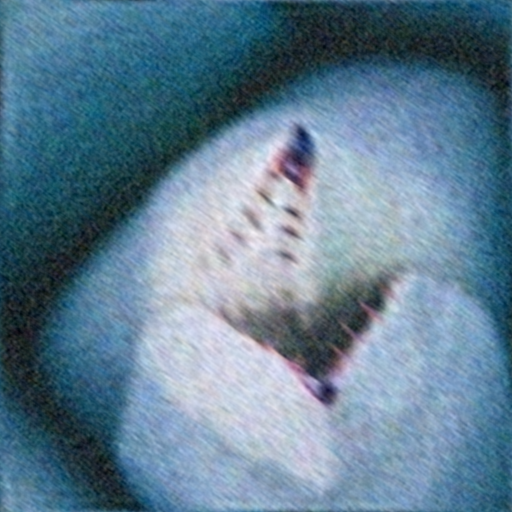}};
			\spy [black, thick] on (0.3,0.3) in node  at (-0.095\textwidth,1.35);
	\end{tikzpicture}}
	
	\subfloat[$\vx^\star$]{\begin{tikzpicture}[spy using outlines={rectangle, magnification=3, size=0.1\textwidth, connect spies}]
			\node {\includegraphics[width=\factor\textwidth]{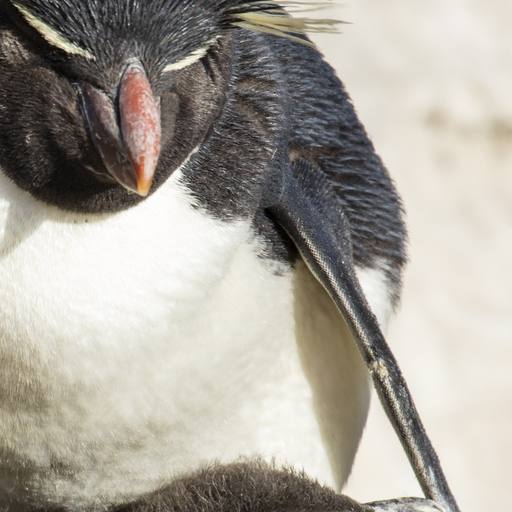}};
			\spy [red, thick] on (-1,1) in node  at (0.095\textwidth,1.35);
	\end{tikzpicture}}\hfill\subfloat[$\vg$; Out--of--Focus Blur]{\begin{tikzpicture}[spy using outlines={rectangle, magnification=3, size=0.1\textwidth, connect spies}]
			\node {\includegraphics[width=\factor\textwidth]{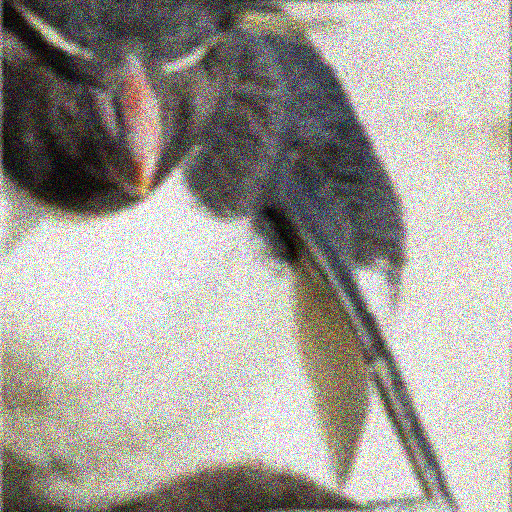}};
			\spy [red, thick] on (-1,1) in node  at (0.095\textwidth,1.35);
	\end{tikzpicture}}\hfill\subfloat[PSNR:23.52]{\begin{tikzpicture}[spy using outlines={rectangle, magnification=3, size=0.1\textwidth, connect spies}]
			\node {\includegraphics[width=\factor\textwidth]{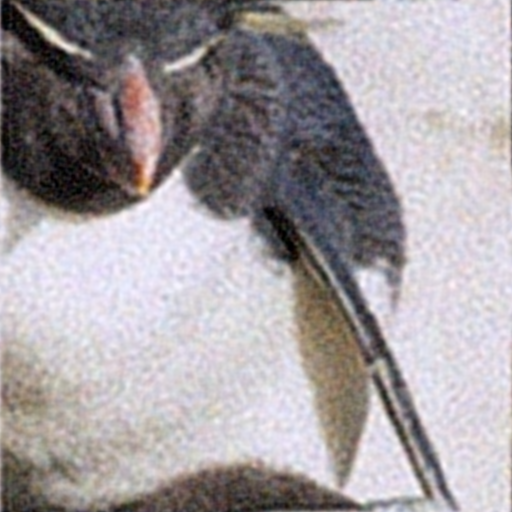}};
			\spy [red, thick] on (-1,1) in node  at (0.095\textwidth,1.35);
	\end{tikzpicture}}

	\caption{First column: clean image. Second column: perturbed data $\vg$. Third column: \gsdsplit reconstruction. The proposed method can handle different type of blurs and noise levels.\label{fig:blurs}}
\end{figure}

\section{Conclusions}
\label{sec:conc}
This work introduces \gsdsplit, an extension of the \pnpsplit splitting
scheme for Poisson image restoration. In the proposed method, the generic
plug-and-play denoiser is replaced by a Gradient Step Denoiser constructed
from a learned convex potential. This choice provides an explicit connection
between the smoothness of the potential, the relaxation parameter $\alpha$,
and the firm nonexpansiveness of the resulting denoiser.

The convergence result follows from the \pnpsplit framework whenever the
denoiser is firmly nonexpansive. In particular, this property is guaranteed
under the sufficient condition $\alpha L\leq 1$, where $L$ is the smoothness
constant of the learned potential. The numerical validation provides
consistent empirical evidence that the trained denoiser operates in an
FNE-compatible regime for $\alpha=.9$ on the considered image datasets.

The numerical experiments show that \gsdsplit provides reliable
reconstructions for Poisson image deblurring. Although \pnpsplit attains an higher reconstruction quality when both methods use their
individually selected values of the ADMM parameter $\gamma$, \gsdsplit
exhibits a lower sensitivity to its choice. Competitive reconstructions are
also obtained for some values of $\alpha$ outside the empirically supported
FNE regime, although such cases are not covered by the convergence analysis. Nonetheless, this higher robustness to the $\gamma$ parameter comes with an higher computational cost, that may be mitigated by a more tailored implementation.

Future work will focus on extending the proposed splitting strategy to RED-PRO formulations for Poisson inverse problems, including constrained models and weaker assumptions than firm nonexpansiveness for the convergence of the resulting algorithms.  A further investigation involves a possible adaptive strategy, and relative convergence analysis, for the relaxation parameter $\alpha$.

\section*{Acknowledgements} This research has been partially performed under CARIPLO Project "Data Science Approach for Carbon Farming Scenarios (DaSACaF)" \verb"CAR_RIC25ABENF_01", and under INDAM-GNCS project "Oltre il gradiente deterministico: analisi di dinamiche stocastiche per funzioni non convesse".

\section*{Usage of AI Tools.} Language editing, stylistic revision, and assistance in drafting and reviewing portions of the code were provided by OpenAI’s GPT-5.6 Thinking model. The author reviewed, tested, and approved all suggestions and remains solely responsible for the manuscript and the accompanying software.

\bibliographystyle{tfs}
\bibliography{GSD}

@InProceedings{pmlr-v70-amos17b,
	title = 	 {Input Convex Neural Networks},
	author =       {Brandon Amos and Lei Xu and J. Zico Kolter},
	booktitle = 	 {Proceedings of the 34th International Conference on Machine Learning},
	pages = 	 {146--155},
	year = 	 {2017},
	editor = 	 {Precup, Doina and Teh, Yee Whye},
	volume = 	 {70},
	series = 	 {Proceedings of Machine Learning Research},
	month = 	 {06--11 Aug},
	publisher =    {PMLR},
}

@ARTICLE{5557884,
	author={Arbeláez, Pablo and Maire, Michael and Fowlkes, Charless and Malik, Jitendra},
	journal={IEEE Transactions on Pattern Analysis and Machine Intelligence}, 
	title={Contour Detection and Hierarchical Image Segmentation}, 
	year={2011},
	volume={33},
	number={5},
	pages={898-916},
	doi={10.1109/TPAMI.2010.161}}

@InProceedings{Agustsson_2017_CVPR_Workshops,
	author = {Agustsson, Eirikur and Timofte, Radu},
	title = {NTIRE 2017 Challenge on Single Image Super-Resolution: Dataset and Study},
	booktitle = {The IEEE Conference on Computer Vision and Pattern Recognition (CVPR) Workshops},
	month = {July},
	year = {2017}
}

@article{aghabiglou2024r2d2,
	title={The {R2D2} deep neural network series paradigm for fast precision imaging in radio astronomy},
	author={Aghabiglou, Amir and Chu, Chung San and Dabbech, Arwa and Wiaux, Yves},
	journal={The Astrophysical Journal Supplement Series},
	volume={273},
	number={1},
	pages={3},
	year={2024},
	publisher={The American Astronomical Society}
}

@article{BENFENATI2026986,
	title = {Unsupervised noisy image segmentation using Deep Image Prior},
	journal = {Mathematics and Computers in Simulation},
	volume = {239},
	pages = {986-1003},
	year = {2026},
	issn = {0378-4754},
	doi = {https://doi.org/10.1016/j.matcom.2025.07.052},
	author = {Alessandro Benfenati and Ambra Catozzi and Giorgia Franchini and Federica Porta},
}

@article{benfenati2025plug,
	title={Plug and Play Splitting Techniques for {Poisson} Image Restoration},
	author={Benfenati, Alessandro},
	journal={Journal of Mathematical Imaging and Vision},
	volume={67},
	number={6},
	pages={59},
	year={2025},
	publisher={Springer}
}

@book{10.1088/2053-2563/aae109,
	author = {Bertero, Mario and Boccacci, Patrizia and Ruggiero, Valeria},
	title = {Inverse Imaging with {Poisson} Data},
	publisher = {IOP Publishing},
	year = {2018},
	series = {2053-2563},
	isbn = {978-0-7503-1437-4},
	doi = {10.1088/2053-2563/aae109},
	address = {Bristol, England, UK}
}

@book{bertero2021introduction,
	title={Introduction to inverse problems in imaging},
	author={Bertero, Mario and Boccacci, Patrizia and De Mol, Christine},
	year={2021},
	publisher={CRC press},
	address = {Boca Raton}
}

@article{booth2006power,
	title={Power iteration method for the several largest eigenvalues and eigenfunctions},
	author={Booth, Thomas E},
	journal={Nuclear science and engineering},
	volume={154},
	number={1},
	pages={48--62},
	year={2006},
	publisher={Taylor \& Francis}
}

@article{boyd2011distributed,
	title={Distributed optimization and statistical learning via the alternating direction method of multipliers},
	author={Boyd, Stephen and Parikh, Neal and Chu, Eric and Peleato, Borja and Eckstein, Jonathan and others},
	journal={Foundations and Trends{\textregistered} in Machine learning},
	volume={3},
	number={1},
	pages={1--122},
	year={2011},
	publisher={Now Publishers, Inc.},
	doi={https://doi.org/10.1561/2200000016}
}

@article{buades2011non,
	title={Non-local means denoising},
	author={Buades, Antoni and Coll, Bartomeu and Morel, Jean-Michel},
	journal={Image processing on line},
	volume={1},
	pages={208--212},
	year={2011},
	doi={1https://doi.org/10.5201/ipol.2011.bcm_nlm}
}

@Article{jimaging10020050,
	AUTHOR = {Benfenati, Alessandro and Cascarano, Pasquale},
	TITLE = {Constrained Plug-and-Play Priors for Image Restoration},
	JOURNAL = {Journal of Imaging},
	VOLUME = {10},
	YEAR = {2024},
	NUMBER = {2},
	ARTICLE-NUMBER = {50},
	PubMedID = {38392098},
	ISSN = {2313-433X},
	DOI = {10.3390/jimaging10020050}
}

@INPROCEEDINGS{9732379,
	author={Cascarano, Pasquale and Sebastiani, Andrea and Comes, Maria Colomba and Franchini, Giorgia and Porta, Federica},
	booktitle={2021 21st International Conference on Computational Science and Its Applications (ICCSA)}, 
	title={Combining Weighted Total Variation and Deep Image Prior for natural and medical image restoration via {ADMM}}, 
	year={2021},
	volume={},
	number={},
	pages={39-46},
	doi={10.1109/ICCSA54496.2021.00016}}

@article{cheng2024diffusion,
	title={A diffusion equation for improving the robustness of deep learning speckle removal model},
	author={Cheng, Li and Xing, Yuming and Li, Yao and Guo, Zhichang},
	journal={Journal of Mathematical Imaging and Vision},
	volume={66},
	number={5},
	pages={801--821},
	year={2024},
	publisher={Springer},
	doi={https://doi.org/10.1007/s10851-024-01199-6}
}

@ARTICLE{10415495,
	author={Cascarano, Pasquale and Benfenati, Alessandro and Kamilov, Ulugbek S. and Xu, Xiaojian},
	journal={IEEE Signal Processing Letters}, 
	title={Constrained Regularization by Denoising With Automatic Parameter Selection}, 
	year={2024},
	volume={31},
	number={},
	pages={556-560},
	doi={10.1109/LSP.2024.3359569}}

@article{chouzenoux2024variational,
	title={A variational approach for joint image recovery and feature extraction based on spatially varying generalised {G}aussian models},
	author={Chouzenoux, Emilie and Corbineau, Marie-Caroline and Pesquet, Jean-Christophe and Scrivanti, Gabriele},
	journal={Journal of Mathematical Imaging and Vision},
	volume={66},
	number={4},
	pages={550--571},
	year={2024},
	publisher={Springer},
	doi={https://doi.org/10.1007/s10851-024-01184-z}
}

@article{cohen2021regularization,
	title={Regularization by denoising via fixed-point projection ({RED-PRO})},
	author={Cohen, Regev and Elad, Michael and Milanfar, Peyman},
	journal={SIAM Journal on Imaging Sciences},
	volume={14},
	number={3},
	pages={1374--1406},
	year={2021},
	publisher={SIAM},
	doi={https://doi.org/10.1137/20M1337168}
}

@article{dabov2007image,
	title={Image denoising by sparse 3-{D} transform-domain collaborative filtering},
	author={Dabov, Kostadin and Foi, Alessandro and Katkovnik, Vladimir and Egiazarian, Karen},
	journal={IEEE Transactions on image processing},
	volume={16},
	number={8},
	pages={2080--2095},
	year={2007},
	publisher={IEEE},
	doi={10.1109/TIP.2007.901238}
}

@article{daniele2026deep,
	title={Deep equilibrium models for poisson imaging inverse problems via mirror descent},
	author={Daniele, Christian and Villa, Silvia and Vaiter, Samuel and Calatroni, Luca},
	journal={SIAM Journal on Imaging Sciences},
	volume={19},
	number={2},
	pages={1077--1109},
	year={2026},
	publisher={SIAM}
}

@article{ducotterd2024improving,
	title={Improving {L}ipschitz-constrained neural networks by learning activation functions},
	author={Ducotterd, Stanislas and Goujon, Alexis and Bohra, Pakshal and Perdios, Dimitris and Neumayer, Sebastian and Unser, Michael},
	journal={Journal of Machine Learning Research},
	volume={25},
	number={65},
	pages={1--30},
	year={2024}
}

@INPROCEEDINGS{9616253,
	author={Dutta, Sayantan and Basarab, Adrian and Georgeot, Bertrand and Kouamé, Denis},
	booktitle={2021 29th European Signal Processing Conference (EUSIPCO)}, 
	title={{Poisson} Image Deconvolution by a Plug-and-Play Quantum Denoising Scheme}, 
	year={2021},
	volume={},
	number={},
	pages={646-650},
	doi={10.23919/EUSIPCO54536.2021.9616253}}

@ARTICLE{8901171,
	author={de Haan, Kevin and Rivenson, Yair and Wu, Yichen and Ozcan, Aydogan},
	journal={Proceedings of the IEEE}, 
	title={Deep-Learning-Based Image Reconstruction and Enhancement in Optical Microscopy}, 
	year={2020},
	volume={108},
	number={1},
	pages={30-50},
	doi={10.1109/JPROC.2019.2949575}}

@inproceedings{
	hurault2022gradient,
	title={Gradient Step Denoiser for convergent Plug-and-Play},
	author={Samuel Hurault and Arthur Leclaire and Nicolas Papadakis},
	booktitle={International Conference on Learning Representations},
	year={2022},
}

@article{klatzer2026efficient,
	title={Efficient {Bayesian} computation using plug-and-play priors for {Poisson} inverse problems},
	author={Klatzer, Teresa and Melidonis, Savvas and Pereyra, Marcelo and Zygalakis, Konstantinos C},
	journal={SIAM Journal on Imaging Sciences},
	volume={19},
	number={2},
	pages={1325--1363},
	year={2026},
	publisher={SIAM}
}

@article{Loreto02012025,
	author = {M. Loreto and T. Humphries and C. Raghavan and K. Wu and S. Kwak},
	title = {A new spectral conjugate subgradient method with application in computed tomography image reconstruction},
	journal = {Optimization Methods and Software},
	volume = {40},
	number = {1},
	pages = {72--95},
	year = {2025},
	publisher = {Taylor \& Francis},
	doi = {10.1080/10556788.2024.2372668},
}

@article{figueiredo2010restoration,
	title={Restoration of {Poisson}ian images using alternating direction optimization},
	author={Figueiredo, M{\'a}rio AT and Bioucas-Dias, Jos{\'e} M},
	journal={IEEE transactions on Image Processing},
	volume={19},
	number={12},
	pages={3133--3145},
	year={2010},
	publisher={IEEE},  doi={10.1109/TIP.2010.2053941}
}

@article{Gao03032020,
	author = {Wenbo Gao and Donald Goldfarb and Frank E. Curtis},
	title = {{ADMM} for multiaffine constrained optimization},
	journal = {Optimization Methods and Software},
	volume = {35},
	number = {2},
	pages = {257--303},
	year = {2020},
	publisher = {Taylor \& Francis},
	doi = {10.1080/10556788.2019.1683553},
}

@article{Setzer2012,
	title = {Deblurring {Poisson}ian images by split {Bregman} techniques},
	journal = {Journal of Visual Communication and Image Representation},
	volume = {21},
	number = {3},
	pages = {193-199},
	year = {2010},
	issn = {1047-3203},
	doi = {https://doi.org/10.1016/j.jvcir.2009.10.006},
	author = {S. Setzer and G. Steidl and T. Teuber},
}

@article{hurault2023convergent,
	title={Convergent {B}regman plug-and-play image restoration for {P}oisson inverse problems},
	author={Hurault, Samuel and Kamilov, Ulugbek and Leclaire, Arthur and Papadakis, Nicolas},
	journal={Advances in Neural Information Processing Systems},
	volume={36},
	pages={27251--27280},
	year={2023}
}

@article{martini2025measuring,
	title={Measuring objective image and video quality: on the relationship between SSIM and PSNR for DCT-based compressed images},
	author={Martini, Maria G},
	journal={IEEE Transactions on Instrumentation and Measurement},
	volume={74},
	pages={1--13},
	year={2025},
	publisher={IEEE}
}

@article{morotti2026adaptive,
	title={Adaptive weighted total variation boosted by learning techniques in few-view tomographic imaging},
	author={Morotti, Elena and Evangelista, Davide and Sebastiani, Andrea and Piccolomini, Elena Loli},
	journal={Journal of Scientific Computing},
	volume={106},
	number={3},
	pages={74},
	year={2026},
	publisher={Springer}
}

@ARTICLE{11299501,
	author={Suzuki, Yodai and Isono, Ryosuke and Ono, Shunsuke},
	journal={IEEE Transactions on Computational Imaging}, 
	title={Convergent Primal-Dual Plug-and-Play Image Restoration: A General Algorithm and Applications}, 
	year={2026},
	volume={12},
	number={},
	pages={142-157},
	doi={10.1109/TCI.2025.3644248}}

@ARTICLE{7936537,
	author={Ono, Shunsuke},
	journal={IEEE Signal Processing Letters}, 
	title={Primal-Dual Plug-and-Play Image Restoration}, 
	year={2017},
	volume={24},
	number={8},
	pages={1108-1112},
	doi={10.1109/LSP.2017.2710233}}

@article{doi:10.1137/20M1387961,
	author = {Pesquet, Jean-Christophe and Repetti, Audrey and Terris, Matthieu and Wiaux, Yves},
	title = {Learning Maximally Monotone Operators for Image Recovery},
	journal = {SIAM Journal on Imaging Sciences},
	volume = {14},
	number = {3},
	pages = {1206-1237},
	year = {2021},
	doi = {10.1137/20M1387961}
}

@article{romano2017little,
	title={The little engine that could: Regularization by denoising ({RED})},
	author={Romano, Yaniv and Elad, Michael and Milanfar, Peyman},
	journal={SIAM journal on imaging sciences},
	volume={10},
	number={4},
	pages={1804--1844},
	year={2017},
	publisher={SIAM},doi={https://doi.org/10.1137/16M1102884}
}

@article{ROND201696,
	title = {{Poisson} inverse problems by the Plug-and-Play scheme},
	journal = {Journal of Visual Communication and Image Representation},
	volume = {41},
	pages = {96-108},
	year = {2016},
	issn = {1047-3203},
	doi = {https://doi.org/10.1016/j.jvcir.2016.09.009},
	url = {https://www.sciencedirect.com/science/article/pii/S1047320316301985},
	author = {Arie Rond and Raja Giryes and Michael Elad},
}

@article{tachella2025deepinverse,
	title = {DeepInverse: A Python package for solving imaging inverse problems with deep learning},
	journal = {Journal of Open Source Software},
	doi = {10.21105/joss.08923},
	year = {2025},
	publisher = {The Open Journal},
	volume = {10},
	number = {115},
	pages = {8923},
	author = {Tachella, Julián and Terris, Matthieu and Hurault, Samuel and Wang, Andrew and Davy, Leo and Scanvic, Jérémy and Sechaud, Victor and Vo, Romain and Moreau, Thomas and Davies, Thomas and Chen, Dongdong and Laurent, Nils and Monroy, Brayan and Dong, Jonathan and Hu, Zhiyuan and Nguyen, Minh-Hai and Sarron, Florian and Weiss, Pierre and Escande, Paul and Massias, Mathurin and Modrzyk, Thibaut and Levac, Brett and Liaudat, Tobías I. and Song, Maxime and Hertrich, Johannes and Neumayer, Sebastian and Schramm, Georg},
}

@INPROCEEDINGS{6737048,
	author={Venkatakrishnan, Singanallur V. and Bouman, Charles A. and Wohlberg, Brendt},
	booktitle={2013 IEEE Global Conference on Signal and Information Processing}, 
	title={Plug-and-Play priors for model based reconstruction}, 
	year={2013},
	volume={},
	number={},
	pages={945-948},
	doi={10.1109/GlobalSIP.2013.6737048}}

@article{zanella2013towards,
	title={Towards real-time image deconvolution: application to confocal and {STED} microscopy},
	author={Zanella, Riccardo and Zanghirati, Gaetano and Cavicchioli, ROBERTO and Zanni, Luca and Boccacci, Patrizia and Bertero, Mario and Vicidomini, Giuseppe},
	journal={Scientific reports},
	volume={3},
	number={1},
	pages={2523},
	year={2013},
	publisher={Nature Publishing Group UK London}
}

@article{zhang2026rethinking,
	title={Rethinking Gradient Step Denoiser: Towards Truly Pseudo-Contractive Operator},
	author={Zhang, Shuchang and Zeng, Yaoyun and Deng, Kangkang and Wang, Hongxia},
	journal={Advances in Neural Information Processing Systems},
	volume={38},
	pages={29585--29612},
	year={2026}
}

@article{Wang2004,
	author = {Zhou Wang and Alan C. Bovik and Hamid R. Sheikh and Eero P. Simoncelli},
	title = {Image quality assessment: From error visibility to structural similarity},
	journal = {IEEE Transactions on Image Processing},
	year = {2004},
	volume = {13},
	number = {4},
	pages = {600--612},
	doi = {10.1109/TIP.2003.819861}
}

@article{zunino2023reconstructing,
	title={Reconstructing the image scanning microscopy dataset: an inverse problem},
	author={Zunino, Alessandro and Castello, Marco and Vicidomini, Giuseppe},
	journal={Inverse Problems},
	volume={39},
	number={6},
	pages={064004},
	year={2023},
	publisher={IOP Publishing},
	doi={10.1088/1361-6420/accdc5}
}

@ARTICLE{7839189,
	author={Zhang, Kai and Zuo, Wangmeng and Chen, Yunjin and Meng, Deyu and Zhang, Lei},
	journal={IEEE Transactions on Image Processing}, 
	title={Beyond a {G}aussian Denoiser: Residual Learning of Deep {CNN} for Image Denoising}, 
	year={2017},
	volume={26},
	number={7},
	pages={3142-3155},
	doi={10.1109/TIP.2017.2662206}}

\end{document}